\documentclass[preprint]{article}

\usepackage{jmlr2e}
\usepackage{lastpage}

\usepackage{amsmath,bm}
\usepackage{booktabs}
\usepackage{graphicx}
\usepackage{float}
\usepackage{mdframed}
\usepackage{tikz}
\usetikzlibrary{arrows.meta,bending,positioning}
\usepackage{microtype}
\usepackage{xcolor}

\newcommand{\X}{X}
\newcommand{\x}{x}
\newcommand{\W}{W}

\newcommand{\Y}{Y}
\newcommand{\Z}{Z}
\renewcommand{\S}{S}
\newcommand{\R}{\mathbb{R}}
\newcommand{\E}{\mathbb{E}}
\newcommand{\etiqueta}{\ell}
\newcommand{\argmin}{\operatorname{arg\,min}}
\newcommand{\p}{p}
\newcommand{\z}{z}
\newcommand{\y}{y}

\newcommand{\s}{s}

\newcommand{\e}{e}
\newcommand{\source}{s}
\newcommand{\perpT}[2]{#1_{#2}}

\newcommand{\defeq}{\stackrel{\text{def}}{=}}

\newcommand{\indep}{\perp\!\!\!\perp}
\DeclareMathOperator{\Var}{Var}
\DeclareMathOperator{\Corr}{Corr}
\DeclareMathOperator{\Cov}{Cov}

\firstpageno{1}
\begin{document}

\title{Invariant Feature Extraction Through Conditional Independence and the Optimal Transport Barycenter Problem: The Gaussian Case}

\author{\name Ian Bounos \email ibounos@dm.uba.ar \\
        \addr Departamento de Matemática, FCEyN, UBA e IMAS, UBA-CONICET
        \AND
        \name Pablo Groisman \email pgroisma@dm.uba.ar \\
        \addr Departamento de Matemática, FCEyN, UBA e IMAS, UBA-CONICET
        \AND
        \name Mariela Sued \email \email msued@udesa.edu.ar\\
        \addr Universidad de San Andrés, CONICET
        \AND
        \name Esteban G.\ Tabak \email egt1@nyu.edu \\
        \addr Courant Institute of Mathematical Sciences, New York University}
\editor{}

\maketitle


    \begin{abstract}
We develop a new methodology for transfer learning in which a predictor trained on a source environment is deployed in a target environment whose joint distribution may differ. The approach extracts $d$ invariant features $W = f(X)$ that predict response variables $Y$ without being confounded by variables $Z$ that may influence both $X$ and $Y$ and whose joint distribution with $Y$ may shift across environments. When the true confounders $Z$ are unknown, observable contextual variables $S$ can be used as surrogates. Motivating applications include medical diagnosis, algorithmic fairness, and domain generalization.
The methodology's main ingredient is the penalization of statistical
dependence between $W$ and $Z$ conditional on $Y$, which we replace by the
more readily implementable plain independence between $W$ and the random
variable $\perpT{Z}{Y} = T(Z,Y)$ that solves the Monge Optimal Transport
Barycenter Problem for $Z \mid Y$. Both
substitutions are \emph{a priori} relaxations,
but in the Gaussian setting considered in this article they turn out to be
equivalences. The resulting linear feature extractor admits a closed form as the leading $d$ eigenvectors of an explicit matrix. 
The same ideas extend to general distributions and nonlinear feature extractors, a setting treated in a companion paper.
\end{abstract}


	\section{Introduction}

    The dominant paradigm in statistical learning has long relied on the assumption that training and test data follow identical distributions. This paradigm has been challenged by learning tasks in which the training data may suffer from various forms of bias, unbalanced/unfair representations and confounding variables. Because of this, it is currently more frequently understood that the population in which one seeks predictions may differ largely from the one used for training.  To highlight this distinction, we denote  the training and deployment distributions, also referred to as \emph{out-of-distribution} (\cite{yang2024generalized, arjovsky2020ood})- as source and target respectively.

	The corresponding challenge is to construct a predictor that, although trained on data from the source distribution, still performs well on the target one, a problem known as  \emph{transfer learning}. Motivational examples abound in the context of medical diagnosis and more general health care settings, where it is commonly observed that algorithms fail to generalize well when applied to populations other than the ones they were originally trained on \cite{Enzo}.

	Guaranteeing a good performance under the target distribution based on knowledge from the source is infeasible without assuming some relation between the two.
	A principle frequently invoked states that, for a relationship between the predictors ${X}$ and the labels $Y$ to be truly informative and stable, it should remain invariant across environmental changes. This idea connects to frameworks such as \textit{causal learning} \cite{fan2024environment}, \textit{invariance risk minimization} \cite{arjovsky2019invariant}, and \textit{domain generalization} \cite{zhou2021domain}
	or  \textit{covariate  shift}, which posit that the conditional distribution of $Y\mid \X$ agrees at source and target.  Others follow  the so called \textit{label shift} invariance, where $X\mid Y$ does not depend on the joint distribution of $(\X,Y)$ but the marginal distribution of $Y$ may vary between source and target.  This condition, often referred to as anti-causal,  reflects situations such as diagnosing a disease ($\Y$) from the symptoms ($\X$) that it originates (see, for instance, \cite{lipton2018detecting} and \cite{scholkopf2012causal}).

\cite{Hickey} propose a Bayesian transfer learning framework,
\emph{Random Effect Calibration of Source to Target} (RECaST), which addresses
distribution shift by recalibrating a source-trained predictor to a target population.
Their approach models discrepancies between source and target through a low-dimensional
random effect and provides uncertainty quantification in the form of
prediction sets with asymptotically correct coverage. In contrast to approaches that seek
invariant representations, RECaST treats transfer learning as a calibration problem at the
prediction stage and is agnostic to the internal structure of the source model.	Recently, \cite{cheng2025transfer} introduced the
\textit{group-label shift assumption}, postulating that the vector of covariates may be partitioned into $(\Z, \X)$  and assume that the conditional distribution of  $\X\mid(\Z,\Y)$ does not depend on the joint distribution of $(\Z,\Y,\X)$, allowing the joint distribution of $(\Z, \Y)$ to differ between in and out of distribution.

The last setting is consistent with the one we consider in this article.  To fix ideas, imagine that $\Y$ is a health outcome and $\X$ represents symptoms, test results, medical images, or biomarkers, while there are other factors, such as patient demographics and medical history, that may influence both $\Y$ and $\X$, yet should not affect the underlying diagnostic relationship we aim to learn, since a) they may not always be available and b) their joint distribution with the outcome $Y$ may depend on the environment.
We subdivide these latter factors into two groups: the \emph{confounders} $\Z$, which together with $Y$ fully determine the distribution of $X$, and any other \emph{contextual} variables $S$, that may relate to $Y$ and $Z$ but do not affect $X$ directly. In these circumstances, it makes sense to only use $\X$ for predicting $\Y$, but to use $Z$ in the learning process, and also the $S$ as surrogates, since at least some of the true confounders $\Z$ are often unobserved or even unknown.

We propose a method to extract features $\W=f(\X)$ that yield robust predictions for $\Y$ by removing the confounding effect of $\Z$. We penalize any dependence between $W$ and $Z$ conditioned on $Y$, formalizing a trade-off between predictive sufficiency and conditional invariance. Both quantities are quantified through an optimal transport perspective and integrated within a single loss function.

For each environment $\e \in \mathcal{E}=\{s,t\}$, where $s$ and $t$ stand for source and target, respectively,   we use $(\Y^\e, Z^e, S^\e,  \X^\e)$, as in the potential outcome literature, to denote the random variables that  correspond to the \emph{same experiments}, that is, they describe the same variables (except possibly for the variables $S$, a set that may depend on the environment) under the different conditions modeled by the environment $\e$.	We use the term \textit{environment} to refer to a collection of circumstances that gives rise to distribution shifts, such as location, time, or experimental conditions, as in \cite{scholkopf2021toward}. Environments encompass both source and target distributions, providing a unified framework to address \textit{Transfer Learning}, \textit{Domain Adaptation}, \textit{Out-of-Distribution Generalization}, and even \textit{Robustness}---different conceptual understandings of very similar problems.

\subsection{Our Contribution}
\label{seq:our_conribution}

We propose to construct features $W = f(\X)$ that satisfy two criteria:
\begin{enumerate}
 \item \label{con_inv}\textbf{Conditional invariance:} $W$ should be independent of  $Z$ conditional on $Y$, i.e., $W \indep \Z\mid Y$.
 \item \textbf{Predictive sufficiency:} $W$ should contain most of the information in $\X$ relevant to predict $Y$.
\end{enumerate}

The condition that $W \indep Z \mid Y$ is referred to as counterfactual invariance in \cite{NEURIPS2021_8710ef76} and is equivalent to $W \mid Y,Z = W\mid Y$. It captures the idea that $W$ should only reflect features of $\X$ that are genuinely related to $Y$, not to spurious correlations with the contextual information. Asking $W$ to be plainly independent of $Z$ would be too restrictive: since $Z$ and $Y$ may be related, any good predictor for $Y$ is likely to be informative on $Z$ too. In our formulation, we penalize any dependence between $W$ and $\Z$ conditioned on $Y$, formalizing a trade-off between predictive sufficiency and conditional invariance in a loss function formulation. Conditional invariance differs from the more conventional requirement that $\Y\mid (\W, \Z) \sim \Y\mid \W$ or $\Y\indep \Z\mid \W$, referred  in \cite{puli2021out} as an uncorrelated representation. Under the latter, if the correlation between $\X$ and $\Z$ depends on the environment, then $\W$ may capture  information related to $\Z$ which can  be harmful for out of sample prediction. 

When addressing fairness, the variable $Z$ can be a set of protected or sensitive attributes (e.g., gender or racial categorization) that we want to be fair with respect to. In this setting, the condition $W\indep Z\mid Y$ is equivalent to {\em equalized odds} \cite{mehrabi2022, hardt2016}. We implement the counterfactual invariance  condition by means of the Optimal Transport Barycenter Problem (OTBP) which, to the best of our knowledge, has not been used before for such a purpose and constitutes one of the main contributions of this work.

In summary, the main contributions of this paper are:

\begin{enumerate}
    \item A new framework to extract features suitable for out of sample predictions, which does not require to assume a causal model nor to know/observe all confounding variables.

	\item A machinery based on the OTBP that on the one hand, gives interpretability, generality and a sound setting to the estimators and, on the other hand, leads to an invariant feature extractor in closed form.

\end{enumerate}


A word is in order on the scope of this paper. We restrict attention to Gaussian distributions and linear feature extractors and predictors, and a reader may reasonably ask why this case deserves a dedicated treatment rather than a section of the forthcoming general paper. Our answer is threefold. First, the methodology proposed here — penalizing conditional dependence on confounders through the optimal transport barycenter, and using contextual surrogates that are in fact equivalent, not merely relaxations, when $\Sigma_{ZS}$ has full rank — is, to the best of our knowledge, new. It has not been developed in the Gaussian case either, despite this being the setting in which invariance, confounding, and transfer learning have been most extensively studied. Second, the Gaussian case is the one in which the machinery becomes fully transparent: conditional independence reduces to a verifiable algebraic condition, the OT barycenter admits a closed form coinciding with a linear regression residual, and the final feature extractor is obtained in closed form as the top eigenvectors of an explicit matrix. This transparency is pedagogically valuable — it lets the reader see exactly what each ingredient of the methodology contributes — and it is also practically valuable, since it yields an estimator that is immediately implementable and interpretable. Third, the same ideas extend to general distributions and nonlinear feature extractors, where the closed-form eigenvector solution is replaced by numerical OTBP solvers and linear maps by functional spaces; this extension is the subject of a forthcoming companion paper. Rather than blurring the two contributions, we have chosen to develop the Gaussian case on its own terms, where the structural properties of the method can be stated and proved cleanly, and to treat the general case separately, where the emphasis shifts to numerical implementation and approximation guarantees.

For conceptual clarity, we develop the methodology at the population level. As mentioned at the end of Subsection~\ref{d1}, the results can be immediately translated to practical situations with a finite number of observations, simply replacing expected values by empirical means.

\subsection{Related work}
\label{subsec:related.work}

In addition to the works mentioned above, several recent approaches have addressed the challenge of learning predictive models that are robust to environmental shifts, especially when covariates $\X$ and outcome $Y$ may be spuriously correlated due to the influence of the environment. These methods attempt to exploit different forms of \textit{invariance} to recover the true signal between predictors and outcomes. We briefly review four prominent approaches: Invariant Risk Minimization (IRM) \cite{arjovsky2019invariant}, Anchor Regression \cite{rothenhausler2021anchor}, Environment Invariant Linear Least Squares (EILLS) \cite{fan2024environment} and Domain-Adversarial Neural Networks (DANN) \cite{ganin2016domain}. Many of these methods build on  the framework presented in \cite{peters2016causal}, where data  $(\X, Y)$ are collected under \textit{different experimental settings or regimes}.  We remark that, since the literature is huge, we cannot provide an exhaustive review and apologize for all resulting omissions.

\cite{arjovsky2019invariant} propose the IRM principle, which aims to learn a data representation $\Phi(\X)$ such that the optimal classifier $w$ on top of $\Phi(\X)$ is invariant across environments. Formally, they seek representations $\Phi$ and classifiers $w$ such that
\[
w \in \argmin_{w'} R^e(w' \circ \Phi) \quad \text{for all environments } e,
\]
where $R^e$ denotes the expected risk in environment $e$. The key idea is that if $\Phi$ captures the \emph{causal features} of the input $\X$—those that genuinely influence the outcome $Y$ in all environments—then the same predictor $w$ will perform well everywhere. In this sense, $\Phi$ is said to be \emph{invariant} because the relationship between $\Phi(\X)$ and $Y$ remains stable, regardless of the environment $e$. In contrast, non-causal or spurious features may vary across environments, leading to predictors that do not generalize.

By enforcing the invariance of the optimal classifier across environments, IRM encourages the learning of representations that exclude spurious correlations and emphasize stable, potentially causal relationships. \cite{rothenhausler2021anchor} introduce Anchor Regression as a principled interpolation between ordinary least squares (OLS) and instrumental variable (IV) regression. Their method leverages exogenous variables called \textit{anchors}, which may influence both $\X$ and $Y$, to formulate a regularized loss:
		\[
		b_\gamma = \argmin_b \mathbb{E}\left[ \left( (I - \Pi_A)(Y - \X^\top b) \right)^2 + \gamma \left( \Pi_A(Y - \X^\top b) \right)^2 \right],
		\]
where $\Pi_A$ denotes projection onto the anchor variables. The parameter $\gamma$ controls the tradeoff between predictive performance on the training data and robustness to perturbations aligned with the anchor variables. This framework is particularly useful when some environmental variation is captured by observed variables $A$, and provides robustness guarantees under specific classes of distributional shifts. \cite{fan2024environment} propose EILLS as a sample-efficient method for linear models across multiple environments. The key assumption is that the conditional expectation $\mathbb{E}[Y \mid  \X_{S^*}]$ is invariant across environments, even if the full joint distribution of $(\X, Y)$ varies. The EILLS objective combines the least-squares risk across environments with a penalty that promotes exogeneity:
\[
Q_\gamma(\beta) = \sum_{e} \mathbb{E}\left[(Y^e - \X^e{}^\top \beta)^2\right] + \gamma \sum_j 1\{\beta_j \neq 0\} \sum_e \left( \mathbb{E}[(Y^e - \X^e{}^\top \beta) X_j^e] \right)^2.
\]
This penalty discourages dependence between the residuals and covariates across environments, effectively selecting features that are invariant and exogenous. Compared to IRM and Anchor Regression, EILLS provides finite-sample statistical guarantees and variable selection consistency, even in high-dimensional settings, without requiring knowledge of anchor variables or structural causal assumptions.

DANN \cite{ganin2016domain} is a deep learning method in which several conceptual features of our approach are already present. In DANN, a feature extractor is trained jointly with a label predictor and a domain classifier, where the latter is connected via a \emph{gradient reversal layer} that encourages the learned features to be invariant with respect to the domain variable $e$. The goal is to extract representations that are useful for predicting $Y$ and do not carry information about $e$.

Our approach departs from DANN in a crucial aspect: we explicitly require that the learned representation $W = f(\X)$ be \emph{conditionally} independent of the contextual variables  $\Z$ given the label $Y$. This is a stronger and more targeted constraint than the marginal invariance encouraged in DANN, and reflects the intuition that features can reflect information about the domain insofar as it is already explained by the label. In contrast, enforcing $W \indep Z$ (which corresponds to the idea used  in DANN) may undesirably remove predictive signal in cases where $Y$ and the environment are themselves highly dependent.

\paragraph{Organization of the paper.} Section \ref{toy_example} presents a simple example to motivate and illustrate the methodology developed in this work.  Section \ref{sec:model} establishes the model we will deal with throughout the paper and the proposal's main idea. Section \ref{sec:methodology} presents the new methodology, starting for expository reasons with the simplest case: one-dimensional response $Y$ and observed confounder $Z$ and joint multivariate Gaussian distribution for $(Y,Z,X)$.  Surrogates $S$ to replace $Z$, higher dimensional $Y$, $Z$ and $W$ and categorical labels are presented in subsequent subsections.   Section \ref{sec:theoretical} establishes the theoretical foundations of the proposal, through a set of proofs that we skip along the exposition to enhance readability. Section \ref{sec:experiments} presents numerical experiments that illustrate and validate the proposal. Section \ref{sec:avenues} discusses briefly possible extensions of the methodology to different scenarios, including general distributions and nonlinear feature extractors, which will be developed more thoroughly in a forthcoming paper.

\paragraph{Notation.} The following notation is used throughout the paper: $\| x\|$ denotes the Euclidean norm of a vector $x$, $A^\top$ stands for the transpose of a matrix $A$, $\text{tr}(A)$ for its trace and $\|A\|_2 = \sqrt{\text{tr}(A^\top A)}$ for its Frobenius norm. The $d$-dimensional identity matrix is denoted by ${\rm I}_d$. Given a random vector $V \in \mathbb R^p$, $\mathbb E(V)\in \mathbb R^p$ denotes its mean value and  $\Sigma_V\in \mathbb R^{p\times p}$ its covariance matrix. If $U\in \mathbb R^s$ is another random vector, $\Sigma_{V U}\in \mathbb R^{p\times s}$ denotes their covariance matrix, with $\Sigma_{V U}(i, j)=\Cov(V_i, U_j)$, while $\Sigma U=\Sigma_{ UU}$. We use $\mu_U(V)$ and $\Sigma_U(V)$ for the conditional mean and variance matrix  of the conditional distribution of $V$ given $U$: $V\mid U$. Finally, the symbols $\indep$ and $\perp$ respectively denote independence and uncorrelation.

\section{Toy example}	\label{toy_example}

Consider the following toy example, which mimics the spirit of Example 1 in  \cite{arjovsky2019invariant}, with structural equation model

\begin{equation*}
		(\Z, \Y) \leftarrow  \mathcal{N}\left(0,\Sigma^e_{ZY}\right),
		\quad
		\X_1\leftarrow   \Z+\mathcal{N}(0, \sigma^2_1),
		\quad
		\X_2\leftarrow  \Y- \Z+\mathcal{N}(0, \sigma^2_2),
\end{equation*}
where $\Sigma^e_{ZY}$ is a covariance matrix for $(\Z, \Y)$ with  $\Var(\Z)=\Var(\Y)=1$ and $\Cov(\Z, \Y)=\rho_e$, making the joint distribution of $(\Y, \Z, \X)$ depend on the environment $e$  through the correlation $\rho_e$.
We observe $(\z_i, \y_i, (\x_1,\x_2)_i)_{i=1}^n$ drawn from a joint source distribution, characterized by $(\Sigma_{\rho_s}, \sigma_1^2,\sigma_2^2)$. The linear combination
		\[
		W = X_1 + X_2 \quad \mbox{satisfies}\quad 	W \;\indep\; Z \mid Y,
	    \]
and is therefore invariant to changes in $\rho$.

By contrast, $X_1$ is predictive of $Y$ only through the confounder $Z$. When $|\rho_s|$ is large, $X_1$ can yield strong predictive performance in source, but that prediction will not generalize to the target environment if this has a different $\rho$, all the way to $\rho=0$, where $Y$ and $Z$ are independent.

This toy example shows how exploiting environment-specific correlations can lead to unstable predictors, while enforcing conditional invariance yields features that generalize across environments. The challenge is to perform this task in real situations, where we do not know the functional relation linking $X$ to $Y$ and $Z$, and we may not observe $Z$ or even know what the true confounder variables $Z$ are.

\section{The model and main ideas}
\label{sec:model}
	
Consider the random vector $(\Y, \Z, \X)\in \R^{d_Y} \times \R^{d_Z} \times \R^{d_X}$ and recall that $(\Y^e, \Z^e, \X^e)$ stands for the outcome of the same experiment at the environment $e\in \mathcal E=\{s, t\}$ (for source and target). We assume that the conditional distribution of $\X$ given $(\Y,\Z)$ does not depend on the environment $e$. Namely, if $\p_e(\y, \z, \x)$ stands for the joint distribution of $(\Y^e,\Z^e, \X^e)$, we assume that it can be factorized as
\begin{equation}
			\label{model}
			\p_e(\y, \z, \x)=\p_e(\y) \;\p_e(\z\mid \y)\;\p (\x\mid \z, \y), \quad\mbox{for any $e\in \mathcal E$.}
\end{equation}
When the marginal for $y$ is independent of the environment, i.e. $p_e(y)=p_t(y)=p(y)$, a family of distributions satisfying this factorization --which does not imply an associated structural model-- is called a nuisance varying model \cite{puli2021out}.  In Subsection \ref{surrogate}, we generalize \eqref{model}, incorporating a contextual variable $S \in \R^{d_S}$, useful when $Z$ is unknown or unobserved

We seek a feature $ W = f(X)\in \R^{d}$ of externally provided dimension $d$, as  informative on $\Y$ as possible but also as independent of $\Z$ as possible \emph{conditioned on $Y$} for the source distribution $p_{\source}$:
\begin{equation}
		W \indep \Z\mid Y,
			\label{W,Z|Y}
\end{equation}
a condition equivalent to both
\begin{equation}
			W \mid Y, Z \sim  W \mid Y  \quad \mbox{and}
			\quad  \Z \mid \Y, W \sim  \Z \mid \Y,
			\label{W|Y,Z}
\end{equation}
\cite{dawid1979conditional},  meaning  that, unlike $\X$, $W$ should not contain more information on $\Z$ than what is already contained in $\Y$. It follows from (\ref{model}) and\eqref{W|Y,Z} that the conditional distribution $W \mid Y$ is independent of the environment. Then it is enough to enforce conditional invariance in the source distribution $\source$, making $W$ well-suited for out-of-distribution predictions. If in addition $p_e(y)$ is independent of the environment, then the full joint distribution of $(W, \Y)$ is invariant across environments. On the other hand, if $p_t(y) \ne p_s(y)$ but $p_t(y)$ is independently known (a case not infrequent in practice), then the joint distribution of $(W, \Y)$ in target can be recovered from the invariant conditional distribution $W \mid Y$.

To quantify the predictive capability of $W$ for $Y$, we consider different criteria depending on the nature of $Y$. For jointly Gaussian variables, we measure the relative error associated with the best linear predictor of $Y$ based on $W$. When $Y$ is categorical, we seek a projection that maximizes the weighted distance between the conditional distributions of $W$ given each class label $Y = y$. Both approaches can be formulated in terms of optimal transport barycenter problems, which facilitates their extension to nonlinear settings. They are presented in detail in subsections \ref{multivariate} and \ref{categorical}.

\subsection{Introducing a framework based on optimal transport}

This subsection shows how the optimal transport barycenter problem, in its Monge formulation proposed in \cite{tabak2021data} and \cite{tabak2025monge}, provides a general framework to enforce conditional invariance and predictive  sufficiency. In order to implement \eqref{W,Z|Y}, we make use of the fact, established in Lemma \ref{relaxed_indep}, that
\begin{equation}
	W \indep Z \mid Y \ \Rightarrow \ W \indep \perpT{Z}{Y},
	\label{one_way}
\end{equation}
where $\perpT{Z}{Y} = T(Z, Y)$ is the solution to the optimal transport barycenter problem
\begin{equation}\label{otbp}
\perpT{Z}{Y}= \argmin_{U= \tilde T(Z, Y)}\; \mathbb{E}\big[c(Z, U)\big]
	\quad \text{s.t. } U \indep Y.
\end{equation}
Here $c(\cdot, \cdot)$ is a pairwise cost function, for which we adopt the squared Euclidean norm
\begin{equation*}
	\label{cost_cuad}
c(x, y) =
\left\|y - x \right\|^2 .
\end{equation*}
Throughout this article, \( \perpT{Z}{Y} = T(Z,Y) \) denotes the solution to the (OTBP) in \eqref{otbp}, interpreted as \( Z \) stripped of the effect of \( Y \): $T$ is the map with minimal cost that removes from $Z$ any variability attributable to $Y$. As discussed in~\cite{tabak2025monge}, for smooth conditional distributions, this problem coincides with the Wasserstein barycenter problem: if $p_y(Z)$ denotes the conditional distribution of $Z \mid Y = y$, then the distribution $p^\ast(\perpT{Z}{Y})$ of the random variable $\perpT{Z}{Y}=T(Z, Y)$ solves the corresponding Wasserstein barycenter problem:
$$
p^\ast=\argmin_p \mathbb E[W_2^2(p_Y, p)], \quad  W_2^2(\rho, \nu)=\inf_{\xi \in \Pi(\rho, \nu)}\mathbb E_{(U,V)\sim\xi}(\vert\vert U-V\vert\vert^2),
$$
where $\Pi(\rho, \nu)$ is the set of joint distributions having $\rho$ and $\nu$ as marginals. Moreover,
\begin{equation*}
	\label{var_dist}
\mathbb E[W_2^2(p_y, p^\ast)]=\mathbb E[ c(Z,	\perpT{Z}{Y})].
\end{equation*}
Then we relax our starting condition \eqref{W,Z|Y} to
\begin{equation}
	W \indep \perpT{Z}{Y},
	\label{WZ}
\end{equation}
much more straightforward to implement and equivalent to \eqref{W,Z|Y} for Gaussian distributions (Lemma \ref{gauss_vuelta}). The optimal transport may be also used to quantify the predictive capability of $W$ for $Y$.
When $Y$ is  continuous, $\perpT{Y}{W}=T(Y, W)$
 quantifies how much $W$ can explain the variability in $Y$ by means of
\begin{equation*}
\mathbb D(Y, W)=\mathbb E[c(Y, \perpT{Y}{W})],
\end{equation*}
since the larger the share of the variability in $Y$ that $W$ can explain, the more costly it becomes to remove from $Y$ that variability by moving $Y$ to $\perpT{Y}{W}$. Thus, high values of $\mathbb D(Y, W)$ are desirable. When the outcome $Y$ is categorical, we can use the cost of the reciprocal optimal transport  $\perpT{W}{Y}=T(W, Y)$,
\begin{equation*}
 \mathbb D(W, Y) = \mathbb E[c(W, \perpT{W}{Y})],
  \label{M_ot_3}
\end{equation*}
which extends the notion of variance to sets of distributions equipped with the Wasserstein distance.

Again, high values of $\mathbb D(W, Y)$ reflect a strong discriminative capability of $W$ for classification. 

Even though this article centers on Gaussian distributions, for which the optimal transport formulation of our problem can be rephrased in more familiar terms associated to linear regression and correlation, the optimal transport barycenter problem still provides a unified conceptual framework, extendable to general distributions.

\section{Methodology for the Gaussian case}
\label{sec:methodology}

In the Gaussian/linear case, one simple objective function with closed form solution covers nearly all instances of the proposed methodology. However, for clarity in the exposition, rather than presenting the full methodology at once, we build it gradually, starting from its simplest instance and incorporating one element of additional complexity at a time. Also for clarity, throughout this gradual development we make a number of statements without proof; these are all proved in Section \ref{sec:theoretical}.

When we omit the superscript $e$ in the random variables $Y^{e}, Z^{e}, X^{e}$, the source environment $s$ is implicitly understood, and all statements involving probabilities, independence or conditional independence, refer to $\p_{s}$. In this section, we assume that under $\p_e$, $(\Y, \Z, \X)$ follows a multivariate  Gaussian distribution.

We seek a linear reduction $W=A^\top (X - \E[X])$, with $A\in \mathbb R^{d_X\times d}$, and a linear prediction for $Y$ based on $W$. Since any invertible function of $W$ is equivalent to $W$ for prediction purposes, we require that $\Cov(W_i, W_j)=\delta_{ij}$.

The Gaussian assumption renders many calculations explicit and, as we shall see, leads to a close expression for $W$. In particular, when $(\Y, \Z)$ is Gaussian, the barycenter $\perpT{Z}{Y}=T(\Z, \Y)$, coincides with the residual of linear regression, except for  the intercept,
 \begin{equation*}
 	\label{T_baticentro}
 	T(\Z,\Y) = \Z - \Sigma_{\Z\Y} \Sigma_\Y^{-1} (\Y - \mathbb E(\Y)) = r(\Z,\Y)+ \mathbb E(\Z) ,
 \end{equation*}
 where  \( r(\Z, \Y) \) is the linear residual of \( Z \) with respect to \( Y \) (Lemma \ref{T.Gauss}). Also, under the Gaussian assumption, the implication in \eqref{one_way} goes both ways:
$$  W \indep Z \mid Y \ \Leftrightarrow \ W \indep \perpT{Z}{Y},  $$
(Lemma \ref{gauss_vuelta}), so using (\ref{WZ}) in lieu of (\ref{W,Z|Y}) involves no relaxation at all.

In order to present the derivations and results in their simplest form, we center and normalize $X$ under $p_{s}$, writing
$\tilde{X} = {\Sigma_X}^{-\frac{1}{2}} (X - \E[\X])$ and $\tilde{A} = {\Sigma_X}^{\frac{1}{2}} A$. Then $W=\tilde{A}^\top \tilde{X}$, and the requirement that $W^\top W = {\rm I}_d$ translates into $\tilde{A} \tilde{A}^\top = {\rm I}_d$. Then, for cleanliness, we
make the replacements $\tilde{A} \rightarrow A$ and  $\tilde{X} \rightarrow X$.

\subsection{One-dimensional response and confounder}
\label{d1}
When both  $\Z$, $Y$  are in $\mathbb R$, we seek a one dimensional reduction $W$.
Then $A \rightarrow a\in \mathbb R^{d_X}$, $W=a^\top X$, $\|a\|^2 = 1$, and the barycenter $\perpT{Z}{Y} = T(Z, Y)$ is given by the residual of linear regression,
$$
\perpT{Z}{Y} = Z - \Sigma_{ZY}\Sigma_Y^{-1} \left(\Y - \E[Y]\right),
$$
where both $\Sigma_{ZY}$ and $\Sigma_Y$ are scalars.
Since $W$ and $\perpT{Z}{Y}$ are jointly Gaussian, their degree of dependence can be  measured by their correlation:
$ W$ is independent of $\perpT{Z}{Y}$ iff $\Corr(W, \perpT{Z}{Y}) = a^\top \Sigma_{X \tilde{Z}} = 0$, where $\tilde{Z} = \Sigma_{\perpT{Z}{Y}}^{-1/2} \perpT{Z}{Y}$, so we seek to minimize $(a^\top   D)^2$ subject to $\left\|a\right\|^2 = 1$, with $D = \Sigma_{X \tilde{Z}}$.

Let us now turn our attention to the requirement that $W$ should be maximally informative on $Y$. In the Gaussian case, maximizing $\mathbb D(Y, W)$ over $W$ is equivalent to minimizing the expected square error of the regression of $Y$ against $W$ (Lemma \ref{R_cuad}). Both yield
\begin{equation*}
\label{C_matrix}
\max_{a} \left(a^\top   C\right)^2, \quad \mbox{where}\quad C = \Sigma_{X Y}.
\end{equation*}
Combining our two goals, predictive sufficiency and conditional invariance, into a single objective function yields the problem
\begin{equation*}
a^* = \arg\max_{\|a\|^2=1} L_\lambda(a)
\end{equation*}
where
\begin{equation}
    L_\lambda(a) = \frac{1 - \lambda}{\Sigma_Y}\ \left(a^\top C\right)^2 - \lambda\  \left(a^\top  D \right)^2 , \quad 0 \le \lambda < 1, \quad C = \Sigma_{X Y}, \quad D = \Sigma_{X \tilde{Z}}.
    \label{Problem_1}
\end{equation}

The parameter $\lambda$ quantifies the relative weight attached to the two goals, with $\lambda = 0$ yielding the OLS solution. The inclusion of the variance $\Sigma_Y$ in the first term of $L_\lambda$ amounts to measuring the regression's relative mean square error, so that the two terms weighted by $\lambda$ and $(1 - \lambda)$ range between zero and one.

The two terms in $L_\lambda$ admit a geometrical interpretation: the first one is maximized when $a$ is aligned with $C=\Sigma_{XY}$, and the second when $a$ is orthogonal to $D=\Sigma_{X\tilde Z}$, so maximizing $L_\lambda$ yields a vector that balances (as quantified by $\lambda$) being parallel to $C$ and orthogonal to $D$. Such vector must lie in the plane spanned by $C$ and $D$, so it is fully determined by one angle.

Better still, problem (\ref{Problem_1}) can be solved in closed form (Lemma \ref{solution}). Its solution $a^\ast $ is the normalized eigenvector corresponding to the largest eigenvalue of the rank two, symmetric matrix
\begin{equation*}
\label{H_matrix}
H = \frac{1 - \lambda}{\Sigma_Y}\ C C^\top - \lambda\ D D^{\top} .
\end{equation*}
Since typically one only has access to a sample $(\y_i, \z_i, \x_i)$, $i=1, \ldots, n$ drawn from the source distribution $\p_{s}$, we use this to estimate the optimal direction through a plug in procedure,  replacing the expected values that define $C$ and $D$ by empirical means.
Having found $a^*$ and hence $W$, we predict out of sample with the coefficient of the linear regression of $Y$ on $W$, once the new target observations $x$ are centered and normalized with the empirical mean and  covariance matrix constructed in source, i.e. with $\x_i$, $i=1, \ldots, n$. See Subsection \ref{pty_new} on how to improve this procedure when the marginal $p_t(y)$ is known.

\subsection{Unknown confounding factor and observed contextual variable}
\label{surrogate}
The procedure developed above uses measurements of the confounding variable $Z$ in the training population. In real scenarios though, one often does not even know what the confounding variables are, forget about measuring them; one is merely aware that such variables may exist. To address this issue, we extend our model to include an observed, \emph{contextual variable} $S$.

For illustration, consider the following triplet $(Z,S,Y)$ and observed features $(X_1,X_2)$:
\[
(Z,S,Y) \leftarrow \mathcal{N}(0,\Sigma^e_{ZSY}),
\quad
X_1 \leftarrow Z + \mathcal{N}(0, \sigma^2_1),
\quad
X_2 \leftarrow Y - Z + \mathcal{N}(0, \sigma^2_2),
\]
with the same structural equations for the observed variables $(X_1, X_2)$ as in Section \ref{toy_example}, but with confounder $Z$ not observed, leaving us only $S$ to use as a surrogate.

In the sequel, we postulate  that
\begin{equation}
    \label{con_S}
    \X\mid (\Y, \Z, S)\sim \X\mid (\Y, \Z)\;, \quad \mbox{ for all possible distributions $p_e$} ,
\end{equation}
i.e.  $S$ can affect $\X$ only through $\Y$ and $\Z$ (This postulate simply extends the definition of $Z$  to include any truly confounding component of $S$.)
Then we assume the following factorization for the joint distribution $\p_e$:
\begin{equation*}
    \label{model.with.s}
    \p_e(\y, \z, s, \x)=\p_e(\y)\; \p_e(\z,s\mid \y)\; \p (\x\mid \z, \y), \quad\mbox{for $e\in \{s, t\}$.}
\end{equation*}
Note that, integrating over $s$, we recover model \eqref{model} for $(Y, Z, X)$. We use $S$ as a surrogate for $Z$, requiring $S$, not $Z$, to be independent of $W$ conditioned on $Y$:
\begin{equation*}
 W \indep S \mid Y.
 \label{W,S|Y}
\end{equation*}
To justify this requirement, notice first that it does not impose any additional constraint on $W$, since, under condition \eqref{con_S},  Lemma \ref{ZtoS} establishes that
$$ W \indep Z \mid Y \ \Rightarrow \ W \indep S \mid Y. $$

Thus the condition that $W \indep S \mid Y$ is a relaxation of $W \indep Z \mid Y$.
Moreover, Lemma \ref{equiv_gauss} shows that if the correlation between $Z$ and $S$ does not vanish, then we also have
$$ W \indep S \mid Y \ \Rightarrow \ W \indep Z \mid Y, $$
making our replacement not a relaxation at all, just an equivalent -but implementable- requirement.

In view of this, it is clear how to proceed. We pose again the problem in (\ref{Problem_1}), with $C=\Sigma_{XY}$  as before but with
$D = \Sigma_{X \tilde{S}}$.
Then the optimal solution $a^\ast$ adopts the same closed form.
 In what follows,  we continue to use the surrogate $S$ in lieu of the confounder $Z$, including in $S$ all observable components of $Z$.

\subsection{Multivariate variables and labels}
\label{multivariate}

Extending the methodology to account for more than one variable of each type, we have $Z \in \mathbb{R}^{d_Z}$, $S \in \mathbb{R}^{d_S}$ and $Y \in \mathbb{R} ^{d_Y}$, while the dimension $d$ of the linear reduction $W$ is chosen by the practitioner. Thus  $W=A^\top  \X$, with $A^\top A={\rm I}_d$.  The reason to consider all components of $Y$ at once and not one at a time, is that we seek a single set of optimal features $W \in \mathbb{R}^{d}$ to predict them all, in a causal instance of dimensionality reduction, with $d \le \min(d_Y, d_X)$. Since $(\Y, S)\in \mathbb R^{d_Y\times d_S}$ is Gaussian, the barycenter $\perpT{S}{Y}=T(S, \Y)$ still agrees with the residual of linear regression.

For scalar $S_Y$ and $W$, we could simply penalize the square of their correlation. Now that they are vectorial though, the pairwise correlation between their components make up a rectangular matrix, and simply adding the squares of its components is not invariant under linear transformations. This is why, for any random vectors $U \in \R^{d_u}$ and $V \in  \R^{d_v}$, we define as a measure of their correlation

\begin{equation*}
 \mathcal{C}(U, V) = \frac{1}{\delta_{uv}} \|\Sigma_U^{-1/2} \Sigma_{UV}\Sigma_V^{-1/2}\|_2^2 , \quad \delta_{uv} = \min\left(d_u, d_v\right),
 \label{Corr}
\end{equation*}
a quantity proportional to the squared Frobenius norm of the covariance between the corresponding standardized variables, which can be easily shown to be invariant under invertible linear transformations of $U$ and $V$. Defining $\tilde{S} = \Sigma_{\perpT{S}{Y}}^{-\frac{1}{2}} \left(\perpT{S}{Y} - \E[\perpT{S}{Y}]\right)$, we have

$$ \mathcal{C}(W, S_Y) =  \frac{1}{\delta_{ws}} \left\|\Sigma_{ W\tilde{S}}\right\|_2^2 = \frac{1}{\delta_{ws}} \ \left\|A^\top D\right\|_2^2, \quad D = \Sigma_{X \tilde{S}}. $$

In order to quantify the predictive capability of \(W\) for \(Y\), we compute \(\mathbb{D}(Y, W)\), which we defined above as the cost of transporting \(Y\) to its barycenter \(Y_W\). The larger \(\mathbb{D}(Y, W)\), the better \(W\) serves as a predictor for \(Y\). As shown in Lemma~\ref{R_cuad}, maximizing \(\mathbb{D}(Y, W)\) over \(W\) is equivalent to minimizing the squared Frobenius norm of the residual from the linear regression of \(Y - \mathbb{E}[Y]\) on \(W\).

Then, for each fixed $\lambda\in (0,1)$, the full problem reads

\begin{equation}
\max_{A^\top A={\rm I}_d}  L (A) = \frac{1 - \lambda}{\text{tr}(\Sigma_Y)} \ \left\|A^\top C\right\|_2^2 - \frac{\lambda}{\delta_{ws}} \ \left\|A^\top D\right\|_2^2,  \quad  C =\Sigma_{X Y}.
\label{L_multivar}
\end{equation}

Theorem \ref{theorem:closed form} solves this problem in closed form: the optimal $A$ has as columns the normalized eigenvectors of $H = \frac{1 - \lambda}{\text{tr}(\Sigma_Y)} \ C C^\top - \frac{\lambda}{\delta_{ws}} \ D D^\top$ corresponding to its $d$ largest eigenvalues.

Notice that we do not need to pick the desired number of features $d$ before hand: we can base this choice on the decay of the eigenvalues of $H$, or alternatively, when some labels for the target distribution are available, on cross-validation, which is rendered particularly economical by the fact that the solution is computed at once for all values of $d$.

\subsection{Categorical labels}
\label{categorical}

We have so far addressed real outcomes $Y$, yet many applications require the prediction of categorical labels, discrete random variables $Y$ taking values in $\{y_1, \ldots, y_k\}$.  It turns out that with only slight variations, the proposed methodology applies to such discrete labels as well. The methodology selects a linear dimension reduction, after which any classification method can be trained with source data, adapted to the target when the marginal distribution $p_t(y)$ is known (Subsection \ref{pty_new}).

The first variation regards the solution of the optimal transport barycenter problem.
When the conditional distributions $S\mid Y=y$ are Gaussian, $S\mid Y=y \sim \mathcal N(\mu_Y, \Sigma_y)$, their barycenter is also Gaussian with mean $\E(S)$ and covariance matrix given by the only positive definite matrix $\Sigma$ satisfying the equation
$ \Sigma=\sum_{y}p_Y(y) (\Sigma^{1/2}\Sigma_y \Sigma^{1/2})^{1/2}$ \cite{agueh2011barycenters, alvarez2016fixed}. The solution to the corresponding optimal transport barycenter problem is
$$ \perpT{S}{Y} = T(S, Y) = \Sigma^{1/2}\Sigma_Y^{-\frac{1}{2}} \left(S - \mu_Y\right) +\mathbb E(S), \quad \mbox{so} \quad \tilde{S} = \Sigma_{S_Y}^{-1/2}(S_Y-\mathbb E[S_Y]) = \Sigma_Y^{-\frac{1}{2}} \left(S - \mu_Y\right). $$

After this, the penalization of conditional dependence between $W$ and $S$ by means of the correlation between $\perpT{S}{Y}$ and $W$  adopts the same form as before,
$$ \|A^\top D \|_2^2, \quad \mbox{with} \quad D =
\Sigma_{\X \tilde{S}} .
$$

A natural approach to quantifying the explanatory power of $W$ on $Y$ is to assess how much the conditional distributions
$W \mid Y = y_j$ differ among the $\{y_j\}$.
To this end, we maximize over $A$ a measure of explanatory power defined by the cost of the OTBP,
$$
\mathbb D(W, Y) = \mathbb{E}\big[c\left(W, W_Y\right)\big],
$$
which extends the notion of variance to the space of probability distributions.

Assuming that $W \mid Y$ has a Gaussian distribution with uniform covariance matrix $\Sigma$, the setting underlying linear discriminant analysis, we have that
$$ W_Y = T(W, Y) = W - \mathbb E[W\mid Y] + \mathbb E[W]$$
(Lemma \ref{T.C.mean}), and so
\begin{equation}
\mathbb D(W, Y) = \sum_j p\left(y_j\right) \big\|\mathbb E\left[W\mid Y=y_j\right] - \mathbb E(W)\big\|^2 = \left\|A^\top C\right\|_2^2,
\label{D_discr}
\end{equation}

where column $j$ of $C$ is given by $ \sqrt{p(y_j)} \left(\mathbb E[X \mid Y = y_j] -\mathbb E(\X)\right)$. Here $p(y)$ stands for $p_t(y)$ when it is known  or the default $p_s(t)$ otherwise. Then,  in this new setting, our problem adopts the same form as before:
$$ \arg\max_{A^\top A={\rm I}_d}  L (A) = \frac{1 - \lambda}{d} \left\|A^\top C\right\|_2^2 - \frac{\lambda}{\delta_{ws}} \ \left\|A^\top D\right\|_2^2, \quad \delta_{ws} = \min(d, d_S), $$
with identical closed form solution.

\subsection{Using knowledge of $p_t(Y)$, the marginal distribution of the response in target.}
\label{pty_new}

Our procedure enforces through penalization, using only source data, the condition that $W \indep Z \mid Y$, which makes the conditional distribution $W \mid Y$ invariant under environmental changes. If the marginal distribution of $Y$ in target, $p_t(y)$, agrees with $p_s(y)$, this is all we need, since then the joint distribution of $(W, Y)$ is the same in source and target.

Frequently though, the marginal distribution of $Y$ does depend on the environment. For instance, the incidence of a particular disease among patients may depend on a hospital's location. If this marginal distribution in target is not known, then there is little that one can do other than assuming by default that $p_t(y) = p_s(y)$. However, $p_t(y)$ is often known, at least to some degree. In the example above, the local incidence of a disease may be known, even though the tests providing the predictor $X$ have not been deployed before in that hospital, only in the hospital corresponding to our source.

The case where $p_t(y) \ne p_s(y)$ is known raises two related questions: how to use $p_t(y)$ to make predictions in target from $W$, and how to tailor the predictive-sufficiency component of the objective function so that $W$ is optimized for predictions in target, not in source.

A first step for making actual predictions, independently of whether $p_t(y)$ is known or not, is to compute $w= W(x)$ for each new target observation $x$. The only subtlety associated with this task is that we have centered and normalized $X$ in the source, so for each observation $x$ from the target distribution, we must define
$$ w = W(x) = A^\top \Sigma_X^{-\frac{1}{2}} \left(x - \E[X]\right).$$
If $p(w \mid y)$ is invariant across environments, we have at the target domain
\begin{equation}
\label{joint_pt}
p_t(w, y) = p_t(y)\; p(w \mid y) = p_s(w,y)\frac{p_t(y)}{p_s(y)}.
\end{equation}
It follows that the marginal distribution $p_t(y)$ in the target set, if known, gives us full knowledge of $p_t(w,y)$.
Moreover, we can transform the available source samples $\{w_i, y_i\} \sim p_s(w, y)$ into virtual weighted samples of the
target distribution $p_t(w, y)$, simply by assigning them the weights $\{\gamma_i = \gamma(y_i)\}$, with $\gamma(y) = p_t(y)/p_s(y)$.

For classification tasks, i.e., for the case of a categorical, discrete label $\Y$,  Bayes rule assigns to each possible observed $w$ the label $\etiqueta(w)$ which maximizes the posterior probability:
\begin{equation}
\label{etiquetas}
\etiqueta(w) = y_{j_*}, \quad \mbox{with} \quad j_* = \arg\max_j\; p(w\mid Y=y_j) p_t(y_j).
\end{equation}
When the conditional distribution $W \mid  Y$ is assumed Gaussian with covariance $\Sigma$ independent of $Y$, as in linear discriminant analysis, i.e., $W\mid Y=y_j\sim \mathcal N(\mu_j, \Sigma)$, $\mu_j=\mathbb E(W\mid Y=y_j)$, then
$$ j_* = \arg\min_j \left\{
\frac{1}{2}
(w-\mu_j)^\top \Sigma^{-1}(w-\mu_j) - \log\left(p_t\left(y_j\right)\right)\right\}. $$
Moreover, we should use the $\{p_t\left(y_j\right)\}$ also to compute $C$ in (\ref{D_discr}), so as to weight the separation among the various $\{\mathbb E\left[W\mid Y=y_j\right]\}$ by their representation in the target, not the source distribution. This way, the very selection of $W$ is tailored to the target distribution to which it will be applied.

Similar considerations apply to the case with continuous  $Y$ and jointly Gaussian distributions.
We seek the linear predictor for $Y$ based on $W$ that minimizes the quadratic mean error in target,
$$
\min_{\beta} \mathbb E_{p_t}\left(\tilde{Y} - \beta \tilde{W}\right)^2 = \mathbb E_{p_s}\left[\left(\tilde{Y} - \beta \tilde{W}\right)^{2} \gamma(Y)\right],
$$
where
$$ \tilde{Y} = Y - \E_{p_s}[\gamma(Y) Y], \quad \tilde{W} = W - \E_{p_s}[\gamma(Y) W],  $$
with solution
\begin{equation}
\label{beta_with_pt}
\beta = \Sigma^\gamma_{YW} {\Sigma^\gamma_W}^{-1}, \quad
\Sigma^\gamma_{YW} = \E_{p_s}\left[\gamma(Y) \tilde{Y} \tilde{W}^\top \right], \quad \Sigma^{\gamma}_W = \E_{p_s}\left[\gamma(Y) \tilde{W} \tilde{W}^\top \right] .
\end{equation}

As we did for categorical labels, we could also use the weights $\{\gamma_i\}$ to tune $W$ itself to be optimal for predicting $Y$ in target: we would need to maximize over $A$
$$ (1 - \lambda) \left\| {\Sigma^{\gamma}_W}^{-1/2}\Sigma^{\gamma}_{WY} \right\|_2^2  - \lambda \| AD \|_2^2 , \quad
W = A^\top X.$$
Unfortunately though, this is a nonlinear problem for $A$, lacking a closed-form solution like the one described before in terms of the eigenvectors of a known matrix, one of the main bonuses of having restricted attention to Gaussian scenarios. In view of this, for continuous labels $Y$, we postpone the development of target-adapted invariant features $W$ to a follow-up article, which will consider general joint distributions and nonlinear functions.

\section{Theoretical grounding}
\label{sec:theoretical}

This section develops the analytical foundation of the methodology, through a sequence of discussions and proofs, some applying to the general problem of label estimation in the presence of confounding variables, others restricted to Gaussian or categorical variables.

\subsection{Conditional independence and contextual variables}

We start with a justification of our procedure's main pillar, which allows us to replace the requirement that $W$ not be confounded by $Z$, i.e. that $W \mid (Y, Z) = W \mid Y$, by the much more readily implementable uncorrelation between $W$ and $S_Y = T(S, Y)$, where the contextual variables $S$ act as surrogates for the confounders $Z$, plain uncorrelation replaces conditional independence and the barycenter $S_Y$ can be computed off-line. The path to this statement comprises a few lemmas.

\begin{lemma}
     $$ W \mid Y, Z = W \mid Y \ \Leftrightarrow \ W \indep Z \mid Y. $$
\end{lemma}

\begin{proof} Observe that,
$$p(w, z \mid y) = p(w \mid z, y) p(z \mid  y) \text{ and } W \indep Z \mid Y \Leftrightarrow p(w, z \mid y) = p(w \mid y) p(z\mid y)$$
so,
$$ W \indep Z \mid Y\ \Leftrightarrow\  p(w \mid z, y) = p(w \mid y). $$
\end{proof}
Rather than assuming knowledge of the confounders $Z$, we have access to contextual variables $S$, which may affect $X$ only through $Y$ and $Z$, i.e. satisfy the condition that $X\indep S\mid (Y, Z)$, so also $W\indep S\mid (Y, Z)$.

\begin{lemma}
\label{ZtoS}
   Let $(Y, S, Z, W)$ be a random vector. If $W\indep S\mid (Y, Z)$ and $W\indep Z\mid Y$, then $W\indep S\mid Y$.
\end{lemma}

\begin{proof}
Because of the  conditional independence between $W$ and $S$ given $(Y, Z)$,
\begin{eqnarray*}
&&\mathbb E(g(W)h(S)\mid Y)=\mathbb E(\mathbb E(g(W)h(S)\mid (Y, Z))\mid Y)=\\ &&\mathbb E(\mathbb E(g(W)\mid (Y, Z))\ \mathbb E(h(S)\mid (Y, Z))\mid Y)=
\mathbb E(g(W)\mid Y)\mathbb E (h(S)\mid Y).
\end{eqnarray*}
\end{proof}
It follows that the condition that $W \indep S \mid Y$ is a relaxation of $W \indep Z \mid Y$. The following lemma quantifies to what degree this relaxes the original constraint when $W$ is linear in $Z$. This is the case in our examples, since we adopt for $W$ a linear function of $X$, and $(X, Z) \mid Y$ is  Gaussian, so $X$ depends on $Z$ only through the conditional expectation $\E[X \mid Z,Y]$, which is linear in $Z$.

\begin{lemma}
\label{equiv_gauss}
If $W = B(Y) Z + \tilde{W}$, where $B(Y) \in \R^{d \times d_Z}$ and $\tilde{W} \in \R^{d}$ is a random variable that may depend on $Y$ but is conditionally independent of $(S, Z)$ given $Y$, we have

\begin{enumerate}

\item For all values of $Y$ for which the nullspace of the conditional covariance $\Sigma_{SZ}(Y) = \mathrm{Cov}(S,Z\mid Y)$ contains only the zero vector,
$$ W \indep S \mid Y \ \Rightarrow \ W \indep Z \mid Y. $$
\item More generally, if $ W \indep S \mid Y$, then $W$ can only be conditionally dependent on the projection of $Z$ onto the nullspace of $\Sigma_{SZ}(Y)$, i.e.
$$ W \indep S \mid Y \ \Rightarrow \ W \indep P(Z) \mid Y, $$
where $P(Z)$ is the projection of $Z$ onto the subspace orthogonal to the nullspace of $\Sigma_{SZ}(Y)$, i.e. onto the range of $\Sigma_{ZS}(Y)$.
\end{enumerate}
\end{lemma}

\begin{proof}

\begin{align*}
W \indep S \mid Y \Rightarrow \quad 0 \
  &= \mathrm{Cov}(S, W \mid Y)
  = \mathrm{Cov}\bigl(S, B(Y) Z + \tilde W \,\big|\, Y\bigr) \\
  &= \mathrm{Cov}(S, B(Y) Z \mid Y) + \mathrm{Cov}(S,\tilde W \mid Y)
  = \Sigma_{SZ}(Y) B(Y)^\top.
\end{align*}
It follows that, if the nullspace of $\Sigma_{SZ}(Y)$ is trivial, $B(Y) = 0$, so $W = \tilde W$,
which is by assumption conditionally independent of $(S,Z)$ given $Y$, so
$$
W \indep S \mid Y \;\Rightarrow\; W \indep Z \mid Y.
$$

More generally, let $ N(Y) := \mathrm{Null}\bigl(\Sigma_{SZ}(Y)\bigr) \subset \mathbb R^{d_Z}$.
Since $B(Y)^\top$ belongs to $\mathcal N(Y)$,
$B(Y)$ annihilates the orthogonal complement $\mathcal N(Y)^\perp$.
Writing
\[
Z = P_{\mathcal N_y}(Z) + (I - P_{\mathcal N_y})Z,
\]
where $P_{\mathcal N_y}$ denotes the orthogonal projection onto $\mathcal N(Y)$, we have
$
B(Y) Z = B(Y) P_{\mathcal N_y}(Z)
$,
and therefore
\[
W = B(Y) P_{\mathcal N_y}(Z) + \tilde W.
\]
Thus $W$ depends on $Z$ only through its projection $P_{\mathcal N_y}(Z)$.
\end{proof}

The requirement that $W \indep S \mid Y$  could be difficult to implement. The following two lemmas justify replacing it by a more easily implementable requirement:

\begin{lemma} Recall that $S_Y = T(S, Y) $. We have,
\label{relaxed_indep}
    $$  W \indep S \mid Y \ \Rightarrow \ W \indep \S_Y.$$

\end{lemma}

\begin{proof}
Since
$W \indep S \mid Y$ and conditional independence is preserved under
measurable transformations of $S$,
\[
W \indep S_Y \mid Y,
\]
so for all bounded measurable functions $f,g$,
\[
\E\big[ f(W)\,g(S_Y) \mid Y \big]
=
\E\big[ f(W) \mid Y \big] \,
\E\big[ g(S_Y) \mid Y \big]
\quad\text{a.s.}
\]
But $S_Y \indep Y$ by definition of the barycenter, so
\[
\E\big[ g(S_Y) \mid Y \big] = \E\big[ g(S_Y) \big]
\quad\text{a.s.},
\]
which replaced into the previous expression and taking its expectation over $Y$ gives
\[
\E\big[ f(W)\,g(S_Y) \big]
= \E\Big[ \E\big[ f(W)\,g(S_Y) \mid Y \big] \Big]
= \E\Big[ \E\big[ f(W) \mid Y \big] \Big] \, \E\big[ g(S_Y) \big].
\]
From the tower property
$
\E\Big[ \E\big[ f(W) \mid Y \big] \Big] = \E\big[ f(W) \big],
$
we obtain
\[
\E\big[ f(W)\,g(S_Y) \big]
= \E\big[ f(W) \big]\,\E\big[ g(S_Y) \big]
\]
for all bounded measurable $f,g$, i.e.
$W \indep S_Y$.
\end{proof}

This lemma shows that the condition that $W \indep S_Y$ is a relaxation of $ W \indep S \mid Y$.
Since solving the optimal transport barycenter problem for $S_Y = T(S, Y)$ involves only the data, not the unknown $A$, it can be performed off-line. Generally, plain independence is easier to enforce than conditional independence. Interestingly, for Gaussian random variables, the reciprocal implication also holds.

\begin{lemma}
\label{gauss_vuelta}
    If $(S, W) \mid  Y$ is Gaussian, then

$$ W \indep S_Y \ \Rightarrow \ W \indep S \mid Y . $$

\end{lemma}

\begin{proof} For Gaussian $S|Y\sim \mathcal N(\mu_{S}(Y),\Sigma_S(Y))$, Lemma \ref{T.C.mean} below provides an  explicit expression for $S_Y$:
\begin{equation}
\label{S_Y_condGauss}
S_Y = T(S, Y) = \left(\Sigma_{S_Y}\right)^{\frac{1}{2}} \left(\Sigma_S(Y)\right)^{-\frac{1}{2}}\left(S - \mu_{S}(Y)\right) + \mathbb E(S).
\end{equation}

Since this is linear in $S$ with $Y$-dependent coefficients, it follows that
\begin{itemize}
\item[(i)] it is enough to prove that $  W \indep S_Y \ \Rightarrow W \indep S_Y \mid Y$.

\item[(ii)] $(S_Y, W) | Y$ is also Gaussian, i.e. $(S_Y, W)\mid Y\sim \mathcal N\left(\mu_{(S_Y,W)}(Y), \Sigma_{(S_Y,W)}(Y)\right)$.
\end{itemize}

Since $S_Y$ and $Y$ are independent, $\mu_{S_Y}(Y)=\mu_{S_Y}=\mu_S$ and
$$
\Sigma_{(S_Y,W)}(Y)
= \begin{pmatrix} \Sigma_{S_Y} & \Sigma_{S_Y,W}(Y) & \\
\Sigma_{W,S_Y}(Y) & \Sigma_W(Y) \end{pmatrix} ,
$$
where
$$
\Sigma_{S_Y,W}(Y)=\Sigma_{S_Y}^{\frac{1}{2}} \Sigma_S(Y)^{-\frac{1}{2}}\Sigma_{S,W}(Y).
$$
Then,
$$
W\mid (S_Y,Y)\sim \mathcal N(\mu_W(S_Y,Y), \Sigma_W(S_Y,Y)) ,
$$
where,
$$
\mu_W(S_Y,Y)=\mu_W(Y)+\Sigma_{WS_Y}(Y)\Sigma_{S_Y}^{-1}(S_Y-\mu_{S_Y}) \defeq \mu_W(Y)+B(Y)\tilde{S}_Y,
$$
with $\tilde{S}_Y = S_Y-\mu_{S_Y}$, and
\begin{eqnarray*}
\Sigma_W(S_Y,Y)&=&\Sigma_{W}(Y)-\Sigma_{WS_Y}(Y)\Sigma_{S_Y}^{-1}\Sigma_{S_Y W}(Y)\\
&=&\Sigma_W(Y)-\Sigma_{WS}(Y)\Sigma_S(Y)^{-1}\Sigma_{SW}(Y) \defeq \Sigma^\ast(Y),
\end{eqnarray*}
depending only on $Y$.

The independence between $S_Y$ and $W$ implies that, for any integrable $g(W)$,
$$
\mathbb E\left(g(W)\mid S_Y\right)=\mathbb E_Y(g(W)\mid Y,S_Y),
$$
so the covariance matrix of $W\mid S_Y$ is given by
$$
\mathbb E_Y\left(\Sigma^\ast(Y) +\{B(Y) \tilde{S}_Y\}\{B(Y) \tilde{S}_Y\}^T \right).$$
Since by assumption $W \indep S_Y$, this conditional covariance may not depend on $S_Y$;  therefore,  ${\rm tr}\left(\mathbb E_Y\left(\{B(Y) \tilde{S}_Y\}\{B(Y) \tilde{S}_Y\}^T \right)\right)=\mathbb E_Y(\vert \vert B(Y) \tilde{S}_Y\vert\vert^2)$ does not depend on $S_Y$ either, which implies that $B(Y)=0$. Then neither $\mu_W(S_Y,Y)$ nor $\Sigma_W(S_Y,Y)$ depend on $S_Y$, which implies that  $W\mid S_Y,Y = W\mid Y $, so $W\indep S_Y\mid Y$.
\end{proof}

This lemma shows that, for conditionally Gaussian variables, replacing $ W \indep S \mid Y$ by $W \indep \perpT{S}{Y}$ involves no relaxation at all. In addition, for jointly Gaussian variables, $\perpT{S}{Y} = T(S, Y)$ is particularly easy to compute:

\begin{lemma}
\label{T.Gauss}
If the joint distribution of $(S, Y)$ is Gaussian, then the solution $\perpT{S}{Y} = T(S, Y)$ of the optimal transport barycenter problem under the quadratic cost $c(x, y) =
\|y - x\|^2$ agrees with the remainder of the linear regression of $S$ against $Y$, except for the intercept,

$$ \perpT{S}{Y} = T(S, Y) = S-\Sigma_{SY}\Sigma_Y^{-1}(Y-\mathbb E(Y)).$$

\end{lemma}

This lemma is a particular case of a more general statement:

\begin{lemma}
\label{T.C.mean}
If the conditional distribution $S|Y$ is Gaussian,
$$ S \mid Y \sim \mathcal N(\mu_S(Y), \Sigma_S(Y)),
$$
then the solution to the corresponding optimal transport barycenter problem with $Y \sim p(Y)$ under the quadratic cost  $c = \|y - x\|^2$ is given by
$$ S_Y = T(S, Y) = \left(\Sigma_{S_Y}\right)^{\frac{1}{2}} \left(\Sigma_Y\right)^{-\frac{1}{2}}\left(S - \mu_Y\right) + \E[S], $$
where $\Sigma_{S_Y} $ is the solution of the fixed point equation
$$ \Sigma_{S_Y} = \mathbb E_{Y}\left[ \left(\Sigma_{S_Y}^{1/2} \Sigma_S(Y) \Sigma_{S_Y}^{1/2} \right)^{1/2}\right] . $$
\end{lemma}

This lemma is proved for discrete variables $Y$ in \cite{alvarez2016fixed}, and extended as Theorem 5 in \cite{YangTabak2020} to general measurable $Y$.

Lemma \ref{T.Gauss} follows, since when $(S, Y)$ are jointly Gaussian, $\Sigma_S(Y)$ does not depend on $Y$, so $\Sigma_S(Y)= \Sigma_{S_Y}$, and $\mu_S(Y) = \mathbb E[S] + \Sigma_{SY}\Sigma_Y^{-1}(Y-\mathbb E(Y))$.

Not only the barycenter agrees with the remainder of linear regression, but the expected value of the squared norm of this remainder, i.e. the regression's mean squared error, equals a constant minus the optimal transportation cost, so minimizing the regression error is equivalent to maximizing the transportation cost. This is the content of the following lemma.

\begin{lemma}
\label{R_cuad}
Let $\Y \in \mathbb R^{d_Y}$ and $\W \in \mathbb R^{d}$ be random vectors, with $\mathbb E[Y] = 0$ and $\mathbb E[W] = 0$, and let $\mathbb D(Y, W)$ be the total transportation cost of the OTBP of $Y | W$. Then,
\begin{equation*}
 \min_{\beta} \mathbb E\left[\left\|Y -\beta W \right\|_2^2\right] = {\rm tr} (\Sigma_Y) - \mathbb D(Y, W), \quad \mbox{which is achieved at  }\beta=\Sigma_{YW}\Sigma_W^{-1}.
\end{equation*}
\end{lemma}

\begin{proof}
Recall that, for any vector $a\in \mathbb R^p$, $\vert\vert a\vert\vert^2 =\text{tr}(a a^\top)$.  Then
\begin{eqnarray*}
\mathbb E\left[\left\| Y -\beta W \right\|^2\right] &=& \mathbb E\left[ \text{tr} \left\{\left(Y -\beta W\right) \left(Y^\top - W^\top \beta^\top\right)  \right\} \right]\\
&=& \text{tr} \left\{ \Sigma_Y - \beta \Sigma_{WY} - \Sigma_{YW} \beta^\top + \beta \Sigma_W \beta^\top \right\} \\
&=& \text{tr} \left\{\Sigma_Y - \Sigma_{YW}\Sigma_W^{-1} \Sigma_{WY} \right\}
\end{eqnarray*}

On the other hand, since $T(Y, W) = Y - \Sigma_{YW}\Sigma_W^{-1} W$,
\begin{eqnarray*}
 \mathbb D(Y, W) &=& \E \left[\|T(Y, W) - Y\|^2 \right] = \E \left[\left\|\Sigma_{YW}\Sigma_W^{-1} W\right\|^2 \right] \\
 &=& \E \left[\text{tr}\left( \Sigma_{YW}\Sigma_W^{-1} W W^\top \Sigma_W^{-1} \Sigma_{YW}^\top \right)\right] \\
 &=& \text{tr}\left( \Sigma_{YW}\Sigma_W^{-1} \Sigma_{WY} \right),
\end{eqnarray*}
completing the proof.
\end{proof}

Since for Gaussian variables, independence is equivalent to uncorrelation, we conclude from this subsection that, when $(W, Z, S)\mid Y$ is Gaussian,
$$ W \mid (Y, Z) = W \mid Y \ \Rightarrow \ \perpT{S}{Y} \perp W , $$
where $\perpT{S}{Y} = T(S, Y)$ has a simple closed-form expression, and
the reciprocal implication also holds whenever $\Sigma_{ZS}$ has full rank. In the latter case, $\perpT{S}{Y} \perp W$ is fully equivalent to the un-confoundedness of $W$, not just a practical relaxation.

\subsection{Closed-form solution to our problem}

We prove a lemma that solves our maximization problem in closed form in terms of the eigenvectors of a matrix.

\begin{lemma} \label{theorem:closed form}
\label{solution}
A solution $A^\ast$ to the problem
$$ \max_{A^\top A = I_{d}} L = \text{tr} \left(\ A^\top H A \right) $$
where $H$ is a symmetric matrix, has as columns the first $d$ normalized eigenvectors of $H$.
\end{lemma}

Note that a more general solution has $A = A^\ast Q$, where $Q$ is an arbitrary orthogonal square matrix. Since $H$ is symmetric but not necessarily positive definite, we define its ``first'' eigenvectors as those with corresponding largest eigenvalues with their sign considered, e.g. $\mu_1 = 2 > \mu_2 = -3$. In order to find them through the power method, it is enough to add to $H$ a multiple of the identity large enough that $H$ becomes non-negative definite.

\begin{proof}
\begin{eqnarray*}
 L &=& \text{tr}\left\{A^\top H A \right\}  \\
 &=& \text{tr}\left[A^\top \left(\sum_j \mu_j h_j h_j^\top \right) A \right],
\end{eqnarray*}
where the $\{\mu_j\}$ are the eigenvalues of $H$ sorted in decreasing order, and the $\{h_j\}$ are the corresponding orthonormal eigenvectors. Then
$$ L = \text{tr} \left(\sum_j \mu_j a_j a_j^\top \right) = \sum_j \mu_j \left\|a_j\right\|^2 , $$
where $a_j = A^\top h_j$ is the projection of $h_j$ on the subspace spanned by the columns of $A$. It follows that $L$ is maximized when this subspace agrees with the one spanned by $\{h_1, \ldots , h_{d}\}$ (so $\|a_j\| = 1$ and $L = \sum_{j=1}^d \mu_j$), proving the claim.
\end{proof}

\section{Experimental validation}
\label{sec:experiments}

This section provides experimental validation of the proposed
barycentric feature extractor.

\subsection{Population experiment for the toy model: covariance shifts and method comparison}
\label{sec:population}

We begin with a population-level experiment designed to compare our proposal, the barycentric
method, with Anchor regression under controlled Gaussian covariance shifts.
Working entirely at the population level allows us to isolate structural effects
from sampling variability and to evaluate both methods on equal footing.

We consider the structural equation model presented in Subsection \ref{surrogate}:
$$
(Z,S,Y) \leftarrow \mathcal{N}(0,\Sigma_\rho),
\qquad
X_1  \leftarrow Z + \mathcal N(0,1),
\qquad
X_2  \leftarrow Y - Z + \mathcal N(0,1).
$$
The Gaussian terms are all independent.
Source and target environments are generated by selecting $\Sigma_\rho$, the covariance matrix for $(\Z, \S,\Y)$, with  $\Var(\Z)=\Var(\S)=\Var(\Y)=1$, $\Cov(\Z, \S)=\rho_{\z,\s}$, $\Cov(\Z, \Y)=\rho_{\z,\y}$, $\Cov(\S, \Y)=\rho_{\s,\y}$,  and $\rho=(\rho_{\z,\s},\rho_{\z,\y},\rho_{\s,\y} )$,  
The parameters $\rho$ are chosen independently and uniformly from a finite grid on the interval $[-1,1]$, discarding those combinations that do not yield positive definite covariance matrices.

For each source--target pair, the penalization parameter in both methods is tuned to minimize the target mean
squared error (MSE).  In this example we  do not use  the target distribution of $Y$ for constructing the linear predictor, i.e., we consider $\gamma(Y)=1$ in Equation \eqref{beta_with_pt}.
If the optimal tuning parameter $\lambda$ is zero, the corresponding method reduces to
ordinary least squares (OLS) {computed at source}.
Otherwise, we report the method achieving the lowest target error.

\begin{table}[ht!]
\centering
\label{tab:counts}
\begin{tabular}{lccc}
\toprule
Method & Anchor & Barycentric & OLS \\
\midrule
Count & 457 & 1020 & 900 \\
Percentage & 19.2\% & 42.9\% & 37.9\% \\
\bottomrule
\end{tabular}
\caption{Number of wins per method (target MSE).}
\end{table}

The barycentric reduction followed by OLS achieves the lowest target MSE in a
substantial fraction of cases, outperforming both Anchor regression (considering $S$ as the anchor) and standard
OLS in the majority of environments involving non-negligible covariance shifts. The best test MSE of Anchor and the barycentric method are compared in Figure \ref{fig:scatterplot-best}.

\begin{figure}[ht!]
    \centering
    \includegraphics[width=0.9\linewidth]{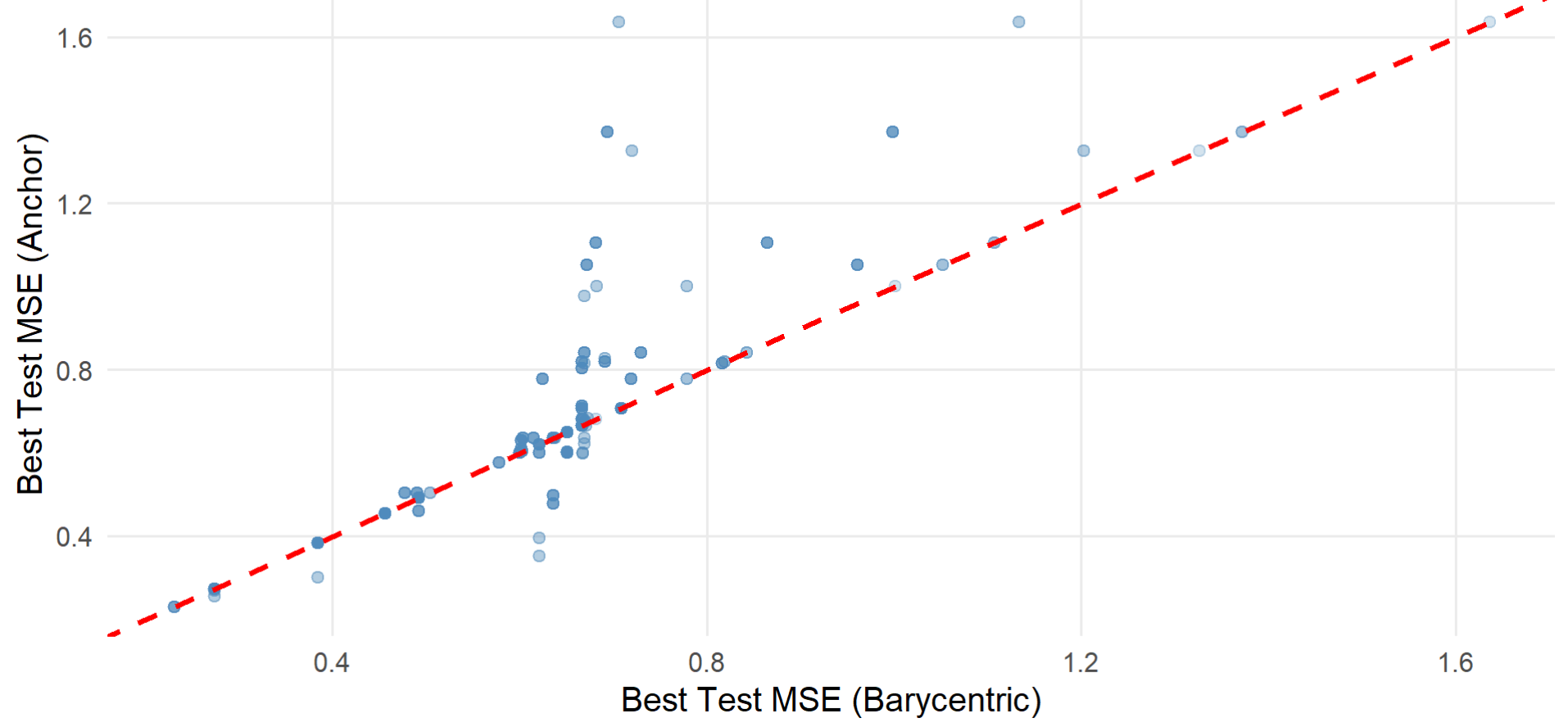}
    \caption{Scatterplot comparing the best target MSE achieved by the
    barycentric method (horizontal axis) and Anchor regression (vertical axis)
    across all valid source--target covariance pairs.
    The dashed diagonal marks the line
    $\mathrm{MSE}_{\text{bar}} = \mathrm{MSE}_{\text{anc}}$.
    Most points lie above this line, indicating that the barycentric reduction
    yields lower target error more frequently.}
    \label{fig:scatterplot-best}
\end{figure}

Figure \ref{fig:frobenius} shows the boxplot of the Frobenius distance between the source and target
covariance matrices, revealing an intuitive structure.
When source and target are close (small Frobenius distance), OLS typically wins,
as expected in nearly i.i.d.\ settings: the additional information in the variables $Z$ overcomes their confounding effect when their joint distribution with $Y$ is similar in source and target.
When environments differ substantially (large Frobenius distance), the
barycentric method tends to dominate, reflecting its advantage in settings with
strong distributional shift.
\begin{figure}[ht!]
\centering\includegraphics[width=0.6\linewidth]{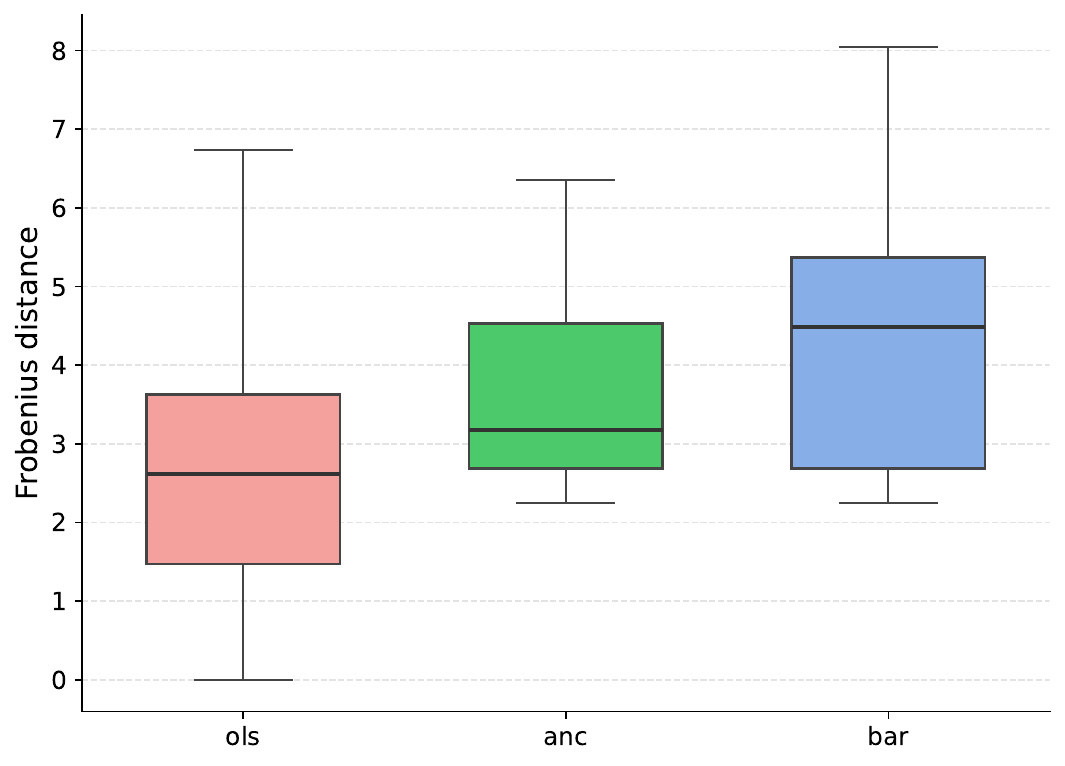}
\caption{Frobenius distance between source and target covariance matrices, grouped by the method achieving the lowest target MSE.
When source and target distributions are very similar (small Frobenius distance), OLS tends to win. As the shift grows, however, the barycentric method---which explicitly penalizes dependence on the
barycentric residual---becomes increasingly advantageous, while Anchor occupies
a small intermediate regime.}
    \label{fig:frobenius}
\end{figure}

\subsection{Interactive Shiny application: the continuous case}

To complement this population-level analysis, we developed an interactive Shiny
application that implements the same structural equation model in a finite-sample
setting, illustrated in Figure~\ref{fig:continuousshiny}. The data-generating
process follows the toy model of the previous subsection but we allow the variances of the error terms to change, i.e.,

$$
(Z, S, Y) \leftarrow \mathcal{N}(0, \Sigma_\rho), \quad X_1 \leftarrow Z + \mathcal{N}(0,\sigma_1^2), \quad X_2 \leftarrow Y - Z + \mathcal{N}(0, \sigma_2^2).
$$

Independent sliders control the source and target correlation structures,
$\sigma_1^2$, $\sigma_2^2$, $\lambda$, the target mean $\mu^t_Y$, the target
variance $\sigma^t_Y$, and a toggle for whether $p_t(Y)$ is known to the
algorithm. The application displays the source and target covariance matrices,
scatter plots of the observed covariates, scatter plots of $Y$ against
$\widehat{Y}(W_{\lambda=1})$ and $\widehat{Y}(W_\lambda)$ in both environments,
and the corresponding relative mean squared errors $\mathrm{MSE}/\mathrm{Var}(Y)$,
allowing the user to track the trade-off between in-sample fit and
out-of-distribution generalization as $\lambda$ varies. The application is
freely accessible at \url{https://2mb35i-ian-bounos.shinyapps.io/OT-invariance/}.
For the example displayed, we have chosen a highly confounding set of parameters, where
a prediction based on $X$ yields predictions nearly opposite to the ground truth. Yet, as the figure shows, 
a prediction based on $W$ with $\lambda = 1$ yields excellent predictions.

\begin{figure}[ht!]
    \centering
 \includegraphics[trim=0 130pt 0 230pt, clip, width=.9\textwidth]{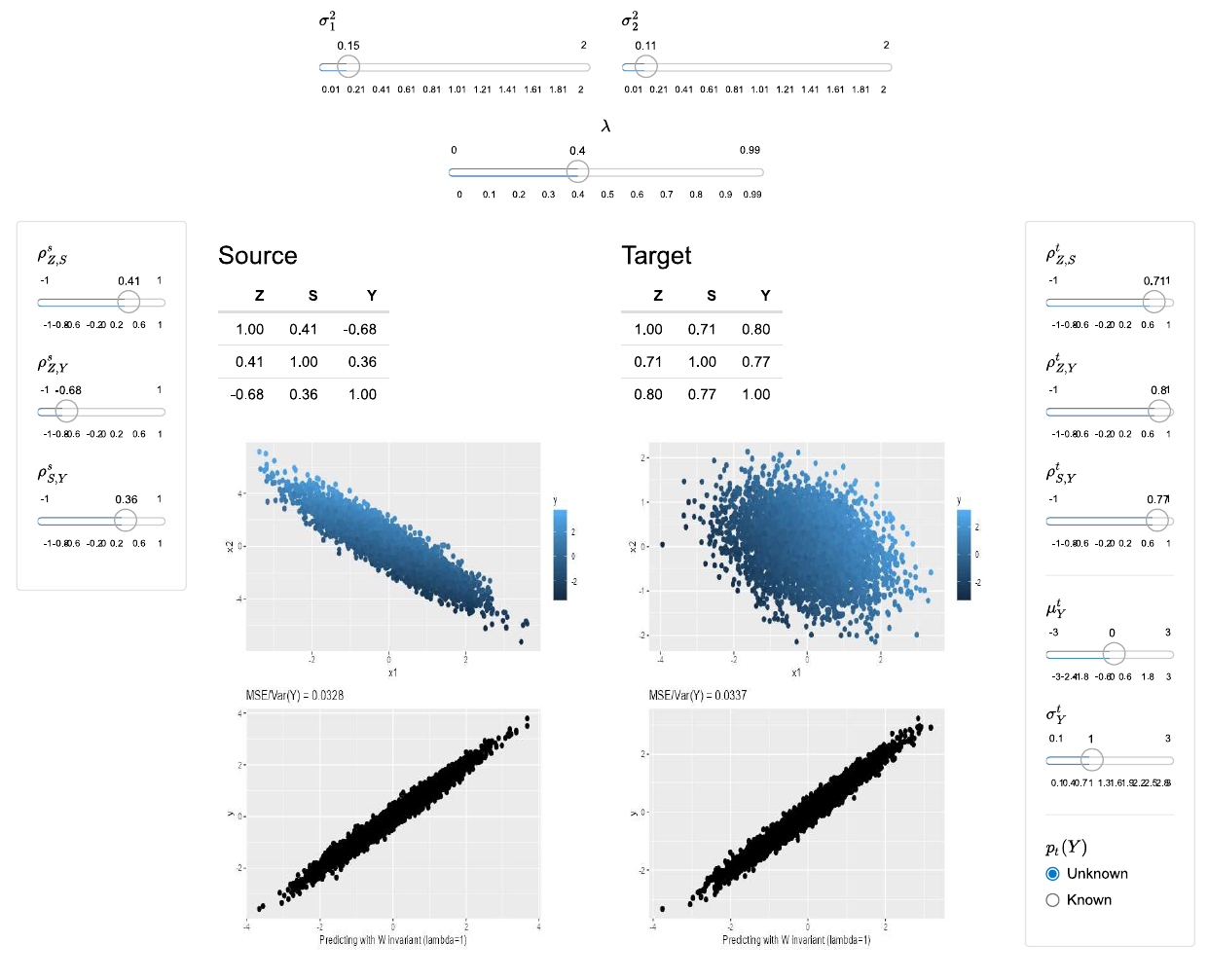}
    \caption{Screenshot of the interactive Shiny application for the continuous-response
case (\url{https://2mb35i-ian-bounos.shinyapps.io/OT-invariance/}). Sliders on the left and right panels control the source and
target environment parameters.}
    \label{fig:continuousshiny}
\end{figure}

\subsection{Multi-dimensional case, continuous response, finite sample}

\paragraph{Experiment A.}
To analyze the case in which the dimensions of $Y$, $S$, and $Z$ are larger than one,
we investigate the structural effect of the regularization parameter $\lambda$
on the invariance properties of the learned representation.
Unlike the previous population experiment, this analysis is conducted in a
finite-sample setting.

Throughout this experiment, we fix the dimensions to
$d_S = d_Z = d_Y = 2$, while the observed feature dimension is set to $d_X = 6$.
For each Monte Carlo repetition, we first generate a random positive definite
covariance matrix $\Sigma^s$ for the joint Gaussian vector $(S,Z,Y)$.
A congruence transformation is then applied to enforce $\operatorname{Cov}(Y)=I$,
thereby ensuring that the distribution of $Y$ is identical across environments,
while preserving arbitrary cross-dependencies with $S$ and $Z$.
An independent covariance matrix $\Sigma^t$ is generated in the same way,
inducing a distribution shift.

Given each covariance, finite samples are drawn for $(S, Z, Y)$  and then $X$ is generated according to the structural model
\[
X \leftarrow A Z + B Y + \mathcal N(0,{\rm I}_{d_X}),
\]
where the matrices $A$ and $B$ are fixed across environments.

For each realization, the proposed method is applied over a grid of $\lambda$ values and we evaluate the conditional correlation between the learned feature representation
and the transformed nuisance variables $T(Z,Y)$.
Figure~\ref{fig:multirep_lambda} reports the results across multiple independent realizations,
showing that the conditional correlation systematically decreases as $\lambda$ increases,
thereby illustrating the effect of the regularization parameter on conditional invariance
in the multivariate, finite-sample regime. The sample size is $m_s=m_t=5000$.

\begin{figure}
    \centering
    \includegraphics[width=0.8\linewidth]{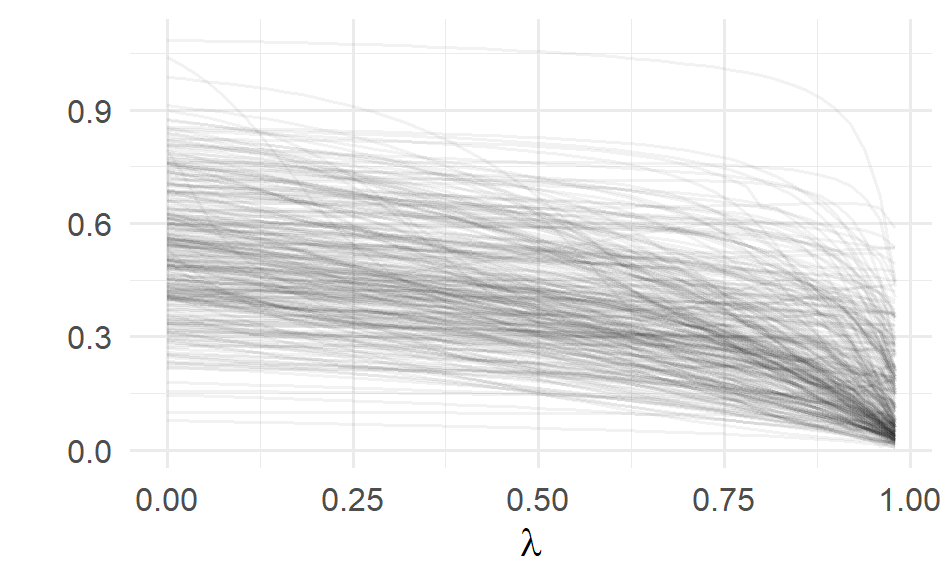}
    \caption{ Conditional correlation curves
    $\lambda \mapsto \|\,\text{Corr}(W_\lambda,Z \mid Y)\|_F$
    over many independently generated environments.
    Despite random variation, all curves exhibit the same structural behaviour:
    a relatively flat region for small $\lambda$ followed by a systematic decay
    and eventual approach to invariance as $\lambda\to 1$.}
    \label{fig:multirep_lambda}
\end{figure}

\paragraph{Experiment B.} In this experiment, we use $ e \in \{s, t\}$, where $s$  and $t$ represent the source and target, respectively.
In order to illustrate the multidimensional Gaussian case, we generate random data in the following way:

\begin{enumerate}

\item We select the problem's dimensions $d_Y=3$, $d_X=20$, $d_Z=4$ and $d_S=5$, the number of source observations $m_s$ and target values $m_t$,  the covariance matrix $\Sigma_Y^{e}$ and expected value $\mu^{e}_Y$, for $e\in \{s,t\}$,  of the marginal distribution of $Y$ in source and target. We also select  matrices $B_Y \in \R^{d_X \times d_Y}$ and $B_Z \in \R^{d_X \times d_Z}$ determining how $X$ depends on $Y$ and $Z$, and three pairs of real parameters $c^{e}_{yz}$, $c^{e}_{ys}$ and $c^{e}_{zs}$, all between zero and one, roughly representing the squared of the pairwise correlations in source and target between $Y$, $Z$ and $S$.

\item We generate $m_s$ and $m_t$ samples from five standard normal vectors: $r^{e}_Y \in \R^{d_Y}$, $r^{e}_Z \in \R^{d_Z}$, $r^{e}_S \in \R^{d_S}$, $r^{e}_X \in \R^{d_X}$ and $r^{e}_{ZS} \in \R^{\min(d_Z, d_S)}$. These are the seeds from which we will build $Y$, $Z$ and $S$.

\item We generate four pairs of random orthogonal matrices $Q^{e}_{zy} \in \R^{d_Z \times d_{y}}$, $Q^{e}_{sy} \in \R^{d_S \times d_{y}}$, $Q^{e}_{zs} \in \R^{d_Z \times d_{zs}}$ and $Q^{e}_{sz} \in \R^{d_S \times \min(d_Z, d_S)}$, where we define a rectangular matrix $Q \in \R^{m, n}$ to be orthogonal when $Q^\top Q = {\rm I}_n$ if $n \le m$ and $Q Q^\top = {\rm I}_m$ otherwise. Introducing a parameter $\alpha$ with $0 \le \alpha \le 1$, we redefine the four target matrices $\{Q^t\}$ into the (no-longer orthogonal) $Q^t \rightarrow \alpha Q^s + (1 - \alpha) Q^t$. This way, the $\{Q^s\}$ and $\{Q^t\}$ range from independently drawn to identical matrices, which, together with the choice of the coefficients $\{c^{e}\}$, allows us to experiment with the degree of difference between the source and target distributions.

\item We set
$$ Y^{e} = \left(\Sigma_Y^{e}\right)^{\frac{1}{2}} r_Y^{e} + \mu_{Y}^{e}, $$
$$ Z^{e} = \left(\Sigma_{\tilde{Z}^{e}} \right)^{-\frac{1}{2}} \left(\tilde{Z}^{e} - \E[\tilde{Z}^{e}]\right), \quad \tilde{Z}^{e} = r^{e}_Z + a^{e}_{zy}\ Q^{e}_{zy}\ r^{e}_Y + a^{e}_{zs}\ Q^{e}_{zs}\ r^{e}_{ZS},  $$
$$ S^{s} = r^{s}_S + a^{s}_{sy}\ Q^{s}_{sy}\ r^{s}_Y + a^{s}_{sz}\ Q^{s}_{sz}\ r^{s}_{ZS},  $$
where $a$ is of the form $\sqrt{\frac{c}{1 - c}}$ for the different values of $c$. We do not need to determine target values for $S$, since these will never be used.

\item Finally, we set the features
$$ X^{e} = r^{e}_X + B_Y Y^{e} + B_Z Z^{e} .$$
\end{enumerate}

We picked $d=3$, seeking as many invariant features $W$ as there are labels $Y$. The matrices $B_Y$ and $B_Z$ were picked at random, with independent entries distributed uniformly in $[-\frac{1}{2}, \frac{1}{2}]$ and then multiplied by $1.5$ and $4.5$ respectively, in order to make the confounding factor $Z$ more relevant to $X$ than the label $Y$. We set the squared correlation-like values $c$ to

$$ c^s_{yz} = 0.8, \quad c^s_{ys} = 0.1, \quad c^s_{zs} = 0.7, \quad
c^t_{yz} = 0.6, \quad c^t_{ys} = 0.3 \quad \hbox{and}\quad  c^t_{zs} = 0.4. $$

The parameters for the source marginal distribution of $Y$ were picked at random, with covariance matrix $\Sigma^s_Y = F^s {F^s}^\top$, with the entries of both the mean vector $\mu_{Y}^s$ and the matrix $F^s$ uniformly distributed in $[-\frac{1}{2}, \frac{1}{2}]$.

We display the results of four typical experiments. In the first two, the marginal of $Y$ in target was made equal to the one in source, so no weights $\gamma(Y)$ are required. For the third and fourth experiment, we adopted instead a different marginal for the target, with
$$ \mu_{Y}^t = \frac{1}{2} \left(\mu_{Y}^s + R \right), \quad \Sigma_Y^t = F^t {F^t}^\top, \quad F^t = \frac{1}{2} \left(F^s + G \right), $$
where the entries of the vector $R$ and the matrix $F^t$ are uniformly distributed in $[-\frac{1}{2}, \frac{1}{2}]$. In the first and third experiments, the parameter $\alpha$ was set to zero, making the joint distributions for $(Y, Z, S)$ independent draws in sample and target, while the second and fourth experiment have $\alpha = \frac{1}{2}$, ensuring certain degree of similarity between the sample and target underlying distributions.

Figure \ref{fig:alpha_0_beta_1} displays the results of the first experiment, plotting the relative mean squared error (MSE) as a function of the penalization parameter $\lambda$, with ordinary least-squares (OLS) for $Y(X)$ used as a reference. At $\lambda = 0$, the results for OLS's $Y(X)$ and our method's $Y(W)$ agree in both source and target: since $d=d_Y=3$, the best choice for $W$ agrees with the subspace of three linear combinations of $X$ that yield $Y$ through OLS. The relative error in target is much bigger than in source, about three times as large, the reason being that OLS made significant use of the relation between $X$ and $Y$ due to the confounding effect of $Z$, which differs significantly in target. As $\lambda$ increases, the accuracy in source deteriorates, since removing the effect of $Z$ takes away a significant source of correlation between $X$ and $Y$. On the other hand, the prediction in target improves, with the two MSE converging to essentially the same number as $\lambda$ approaches one: with the effects of $Z$ removed, the relation between $W$ and $Y$ is identical in source and target. That using the surrogate $S$ worked so perfectly to remove the confounding effects of $Z$, is due to the fact that $d_S \ge d_Z$, so $\Sigma_{ZS}$ has full rank.

The results of the second experiment, displayed in Figure \ref{fig:alpha_0.5_beta_1}, are similar, with a significant difference: since now there is some commonality in the way that $Z$ relates to $Y$ in source and sample, removing the effects of $Z$ completely is not optimal, and the best prediction in target is achieved for $\lambda = 0.84$. For the same reason, even the difference in OLS accuracy between source and target, though still large, is not as extreme as in the first experiment.

In the third experiment (Figure \ref{fig:alpha_0_beta_0.5}), the marginals for $Y$ are not the same in source and target, so using regression results from the former brings some systematic error to the predictions of the latter, though the results for $\lambda \approx 0.9$ are still quite good. Bringing knowledge of the marginal in target through the weights $\gamma(Y)$ in equation \eqref{beta_with_pt}, on the other hand, yields much better answers, and the accuracy in source and target nearly agree for $\lambda = 1$.

This is still the case for the fourth experiment (Figure \ref{fig:alpha_0.5_beta_0.5}). However, because of the common traits in the correlation between $Y$ and $Z$ in source and target, the optimal prediction is achieved at $\lambda = 0.88$ (using the weights.)  Notice that, in the third and fourth experiment, the results of the weighted version of the algorithm do not agree at $\lambda = 0$ with the also weighted OLS. The reason is that, in the presence of weights, OLS will not choose the same linear combination of the $X$ that gave rise to $W$. As the contrasting results of the two experiments show, using the weights for $X$ for small values of $\lambda$ is not necessarily beneficial: since, unlike $W | Y$ when $\lambda =1$, the conditional distribution for $X | Y$ is not invariant across environments, weighting the source samples $\{x_i, y_i\}$ by $\{\gamma_i\}$ does not really produce synthetic samples from the unknown $p_s(X, Y)$.

\begin{figure}[ht!]
    \centering
    \includegraphics[width=0.5\linewidth]{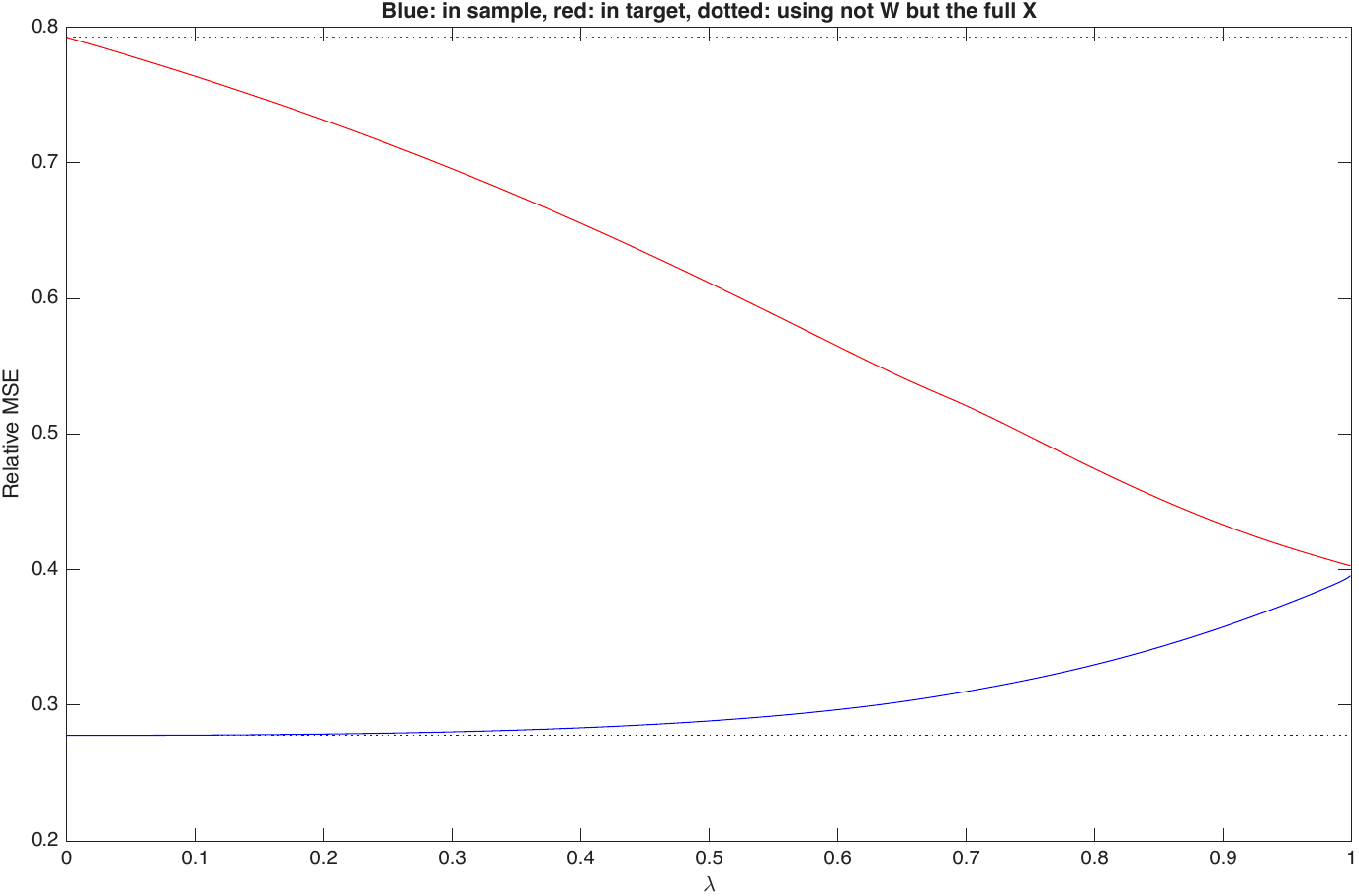}
    \caption{Same marginal distributions in source and target, independently drawn conditional distributions for $(Z, S)$.}
    \label{fig:alpha_0_beta_1}
\end{figure}

\begin{figure}[ht!]
    \centering
    \includegraphics[width=0.5\linewidth]{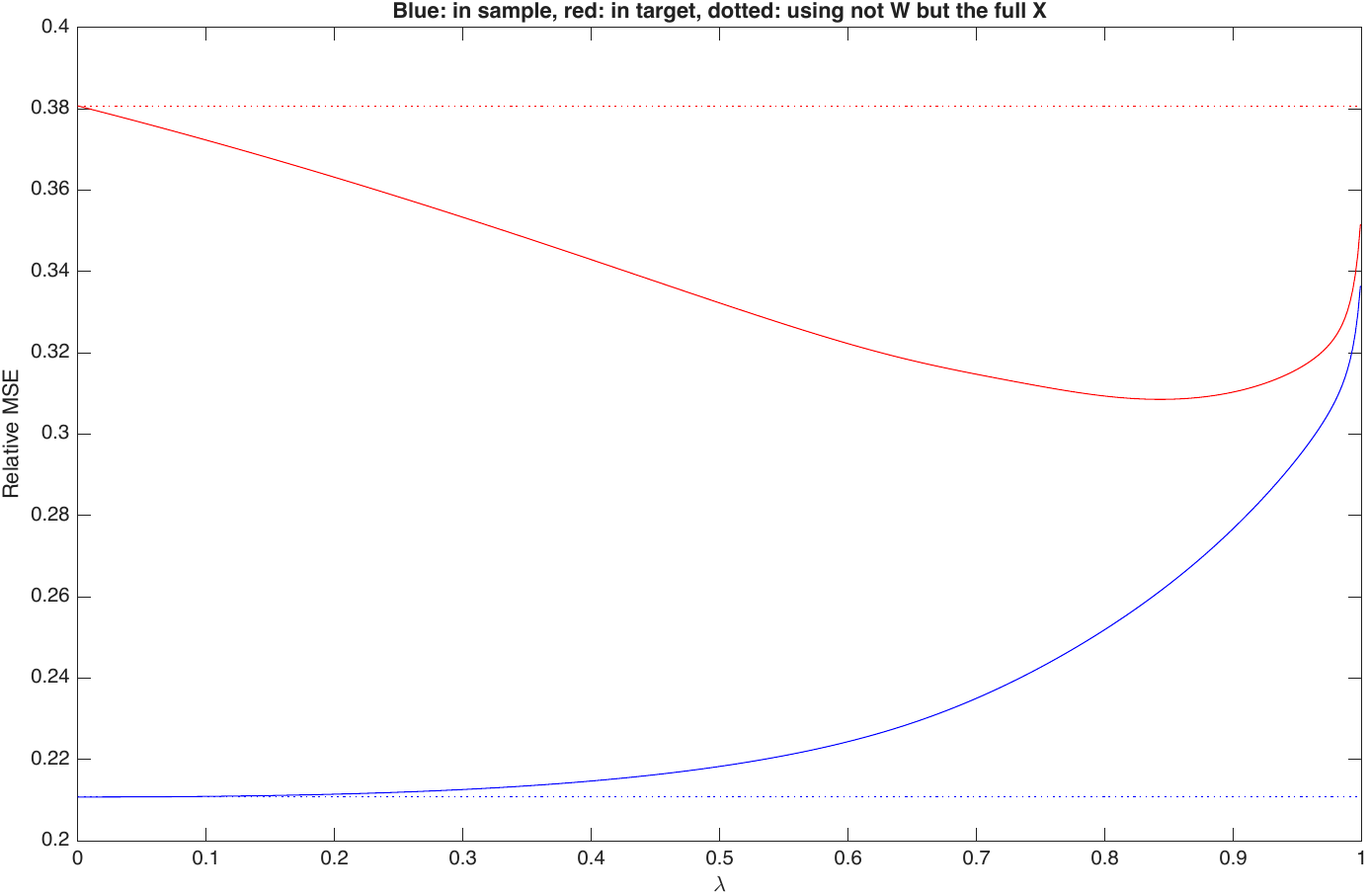}
    \caption{Same marginal distributions in source and target, related conditional distributions for $(Z, S)$.}
  \label{fig:alpha_0.5_beta_1}
\end{figure}

\begin{figure}[H]
    \centering
    \includegraphics[width=0.7\linewidth]{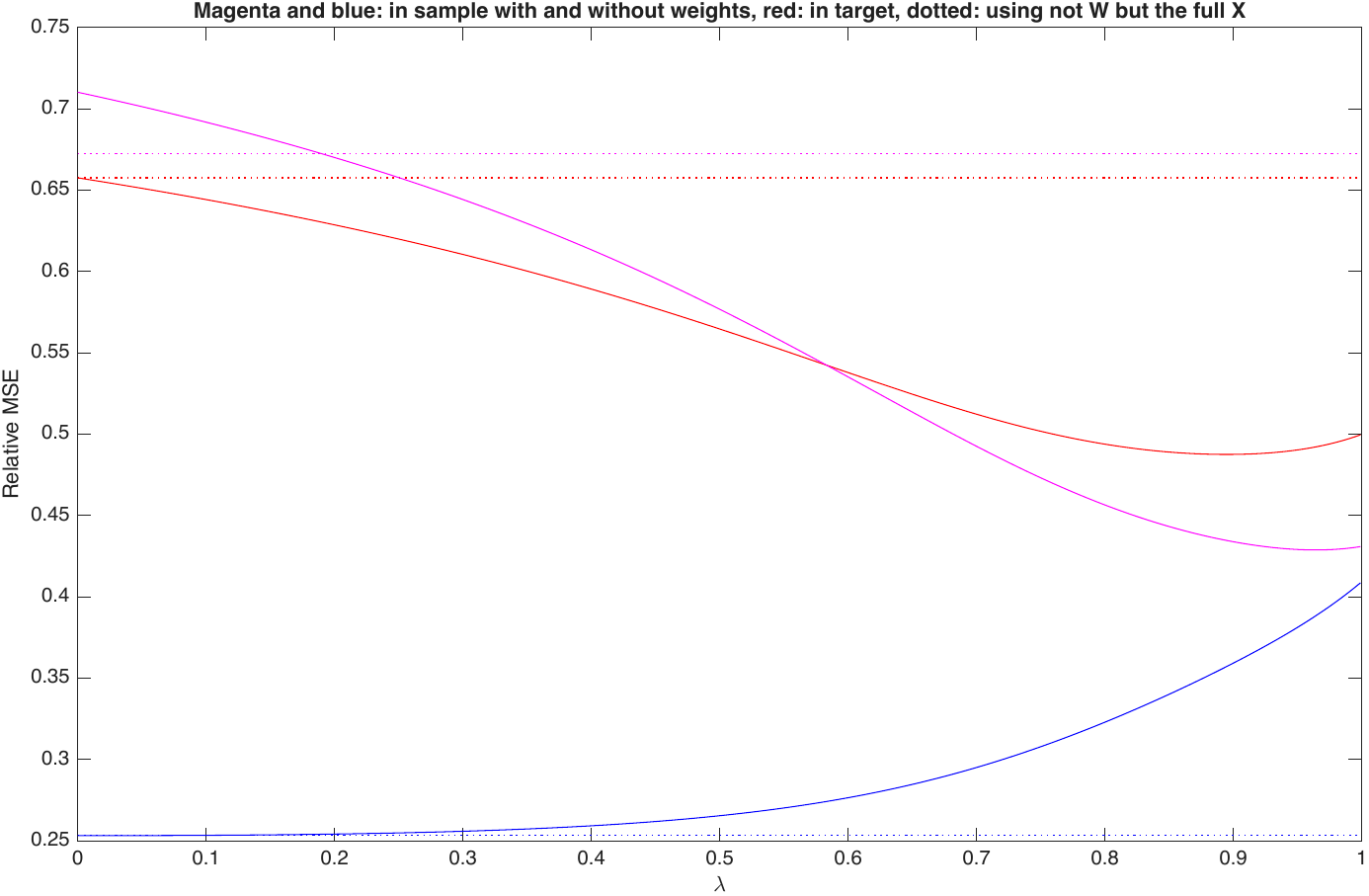}
    \caption{Different marginal distributions in source and target, independently drawn conditional distributions for $(Z, S)$.}
\label{fig:alpha_0_beta_0.5}
\end{figure}

\begin{figure}[H]
    \centering
    \includegraphics[width=0.7\linewidth]{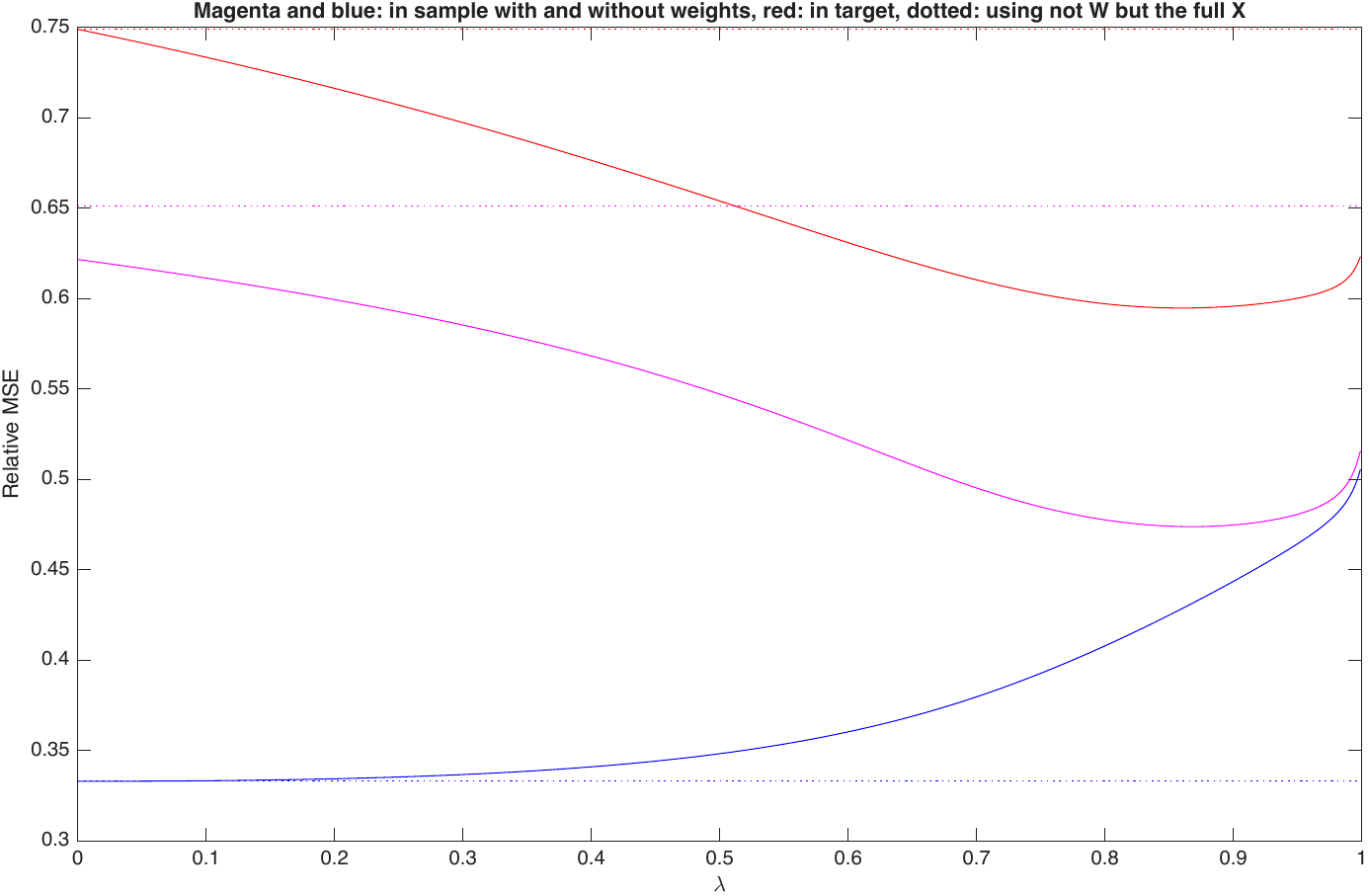}
    \caption{Different marginal distributions in source and target, related conditional distributions for $(Z, S)$.}
\label{fig:alpha_0.5_beta_0.5}
\end{figure}

\subsection{Interactive Shiny application: categorical case}

The categorical-label methodology of Subsections~\ref{categorical}
and~\ref{pty_new} is also implemented in the Shiny application at
\url{https://2mb35i-ian-bounos.shinyapps.io/OT-invariance/}.
As illustrated in Figure~\ref{fig:discreteshiny}, sliders on the left and right
panels control the full set of source and target parameters. All outputs update
in real time and include scatter plots of $(X_1, X_2)$ and of the learned feature
$W_\lambda$ colored by class label, kernel density estimates of
$W_\lambda \mid Y = 0$ and $W_\lambda \mid Y = 1$, a summary table of LDA
target accuracy for $X$, $W_{\lambda=1}$ and $W_\lambda$, and the accuracy
curve $\lambda \mapsto \mathrm{Accuracy}_{\mathrm{target}}(W_\lambda)$.
An optional toggle activates the known-target-marginal regime of
Subsection~\ref{pty_new}, incorporating the importance weights
$\gamma(Y) = p_t(Y)/p_s(Y)$ into both the feature extraction and the LDA
classifier. Again, we have chosen to display an example where naive predictions
based on $X$ (i.e. with $\lambda = 0$, and even up to $\lambda = 0.5$) would systematically yield the wrong binary answer.
Yet as $\lambda$ approaches $1$, $W$ becomes as good a predictor for $Y$ is target as in sample, yielding a high rate of success.

\begin{figure}[H]
    \centering
 \includegraphics[trim=0 200pt 0 228pt, clip, width=1\textwidth]{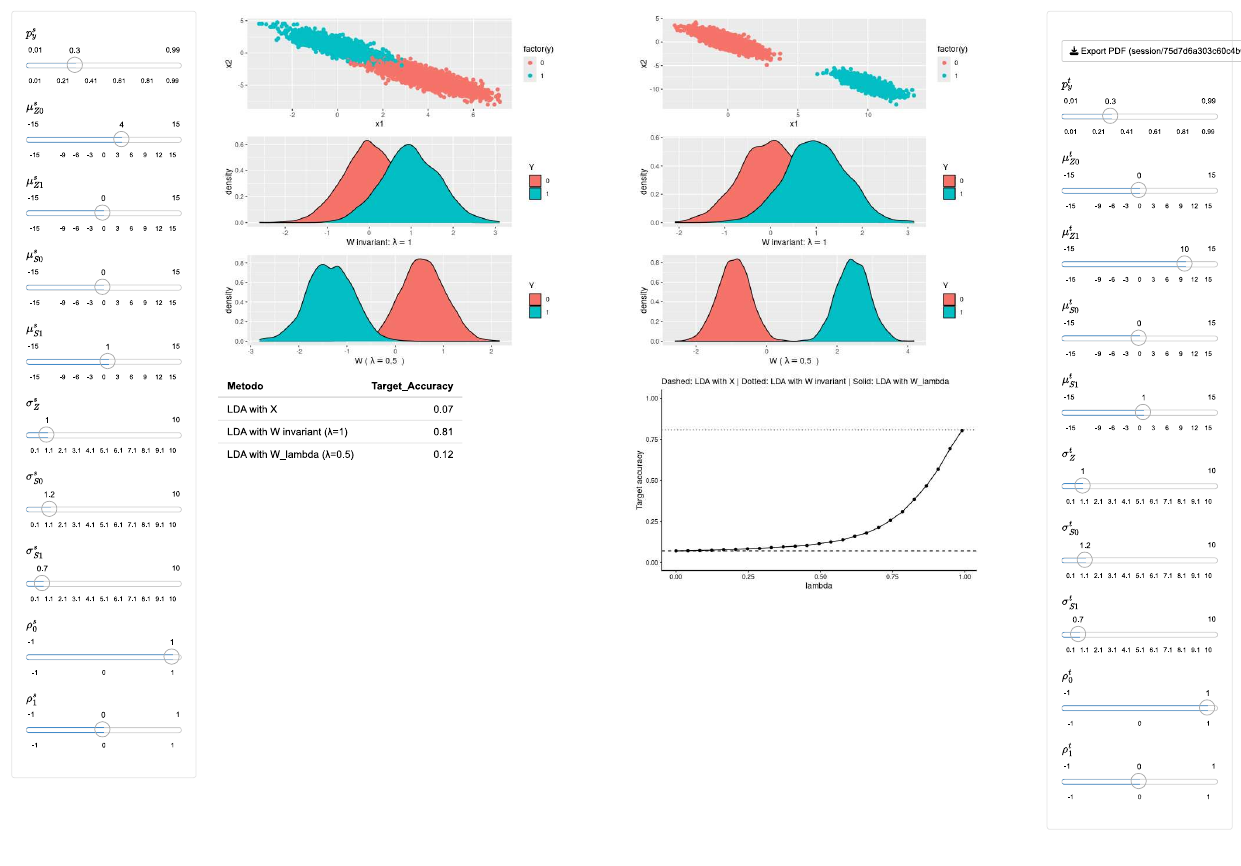}
    \caption{Screenshot of the interactive Shiny application for the categorical-response
case (\url{https://2mb35i-ian-bounos.shinyapps.io/OT-invariance/}). Sliders on the left and right panels control the source and
target environment parameters.}
    \label{fig:discreteshiny}
\end{figure}

\section{Avenues for further development}
\label{sec:avenues}

The methodology developed in this article can be extended to more general scenarios, going well beyond the transfer learning between two environments of linear regression and classification under Gaussian assumptions. We sketch here briefly some of these extensions, which are the subject of current work.

\subsection{Categorical contextual variables}

In many applications, some or all available contextual variables $S$ are categorical, with values in $\{s_1, \ldots, s_k\}$. In order to use categorical contextual variables within this article's procedure, we need to assign meaning to the barycenter problem for $\perpT{S}{Y} = T(S, Y)$, whose regular formulation requires $S$ to live in a smooth manifold. A natural way to do this is to replace the $\{\s_j\}$ by the corresponding probabilities $\{p_j = p(s_j)\}$, and
embed them either in $\R^k$ or in the $(k-1)$-dimensional simplex.

\subsection{Multiple source environments}

This article considered a single source environment $s$, potentially different from the target environment $t$ to which we would like to transfer the predicting procedure for $Y$. Often though, one has access to various source environments $\{e_k\}$, such as different hospitals, regions, times and populations. The resulting additional richness in the data can be exploited to better transfer learning to the target environment. Differences with the case with a single source include the following.

\begin{enumerate}
    \item The very existence of various source environments may do part of the de-confounding work for us, since OLS itself --or any other tool for estimating $Y$ from $X$-- will tend to disregard those features that do not correlate with $Y$ consistently across environments.

    \item The contextual variables $S$ available need not be the same for all environments, as the measurements available in each may differ.

    \item We may relax the assumption that $\p(\X\mid\Z, \Y)$ must be the same for source and target, since having various environments at our disposal allows us to treat the environment $e$ itself as an additional confounding variable,
    i.e. write $\p_e(\X\mid\Z, \Y) = \p(\X\mid\Z, e, \Y)$.

    \item One can use the multiple source environments to perform cross-validation, so as to get  hold on the procedure's only tunable parameter $\lambda$.
\end{enumerate}

\subsection{Extension to general distributions and nonlinear invariant feature extraction}

This article has focused on Gaussian scenarios, where independence relaxes to uncorrelation and linear predictors suffice. Yet the main ideas and tools of the methodology, such as conditional independence and the removal of explainable variability through the OTBP, are by no means restricted to Gaussian distributions and linear feature extractors. Some extensions of the methodology are straightforward. For instance, general nonlinear functions $W = f(X)$ can be written as linear functions of a set of features depending nonlinearly on $X$, as in Reproducing Kernel Hilbert spaces, and independence between variables can be phrased as uncorrelation between arbitrary measurable functions of those variables. The OTBP does not typically admit a closed form solution outside of the Gaussian realm, but effective numerical procedures exist \cite{peyre2019computational,tabak2025monge} and are easy to use, since $\perpT{S}{Y} = T(S, Y)$ depends only on the data and can therefore be performed off-line.

\subsection{Invariant conditional density estimation}

Within this article, the extracted feature $W$ was used to predict $Y$ through either regression for smooth outcomes or classification for categorical labels. However, one of the main motivations behind the development of the Monge OTBP was to go beyond regression, estimating or simulating the full conditional distribution for $Y\mid W$ \cite{YangTabak2020} . It is therefore natural to extend this capability to account for confounders and contextual variables, yielding a form of invariant conditional density estimation.

\section{Conclusion}

This article introduced a feature extraction principle based on conditional invariance,
requiring the learned representation to be independent of contextual variables
given the outcome.
This criterion formalizes the idea that robust predictors should rely on
information genuinely related to the target, rather than on
environment-specific correlations induced by confounding variables.

We showed that this principle can be implemented through an optimal transport
barycenter construction, which leads in the Gaussian–linear case to a simple
closed-form solution.
The resulting barycentric reduction yields features whose joint distribution
with the outcome is invariant across environments, while preserving predictive
content.

Both population-level and sample-based experiments were performed on synthetic data, demonstrating the methodology's out-of-sample predictive power and robustness. Experiments
under simple controlled covariance shifts illustrate how
the proposed approach compares with anchor regression.
The barycentric method is generally advantageous when source and
target environments differ substantially, while reducing to ordinary least
squares in nearly i.i.d.\ settings.

While this work focuses on the Gaussian–linear case, the formulation in terms of
conditional invariance and optimal transport barycenters suggests a natural extension
to more general frameworks, including categorical covariates, non-Gaussian distributions,
nonlinear feature extractors and invariant conditional density estimation.

\vskip 0.2in

\bibliography{references}

\end{document}